\documentclass[a4paper,12pt]{amsart}
\usepackage[T1]{fontenc}
\usepackage{graphicx} 
\usepackage{amsaddr}
\usepackage{color}
\usepackage{layout}
\usepackage{enumerate}
\usepackage[mathlines]{lineno}
\usepackage[colorlinks,linkcolor=red,citecolor=blue]{hyperref}
\usepackage{mathrsfs}
\usepackage{xcomment}
\usepackage{comment}
\usepackage{mathtools}
\mathtoolsset{showonlyrefs}

\RequirePackage[normalem]{ulem} 
\RequirePackage{color}\definecolor{RED}{rgb}{1,0,0}\definecolor{BLUE}{rgb}{0,0,1} 

\usepackage{amsmath}
\usepackage{amssymb}
\usepackage{amsthm}

\newcommand{\R}{\mathbb{R}}

\newcommand{\argmax}{\text{argmax}}

\newcommand{\seq}[2][n]{(#2_{#1})_{#1\geq 1}}

\excludecomment{TAKEOUT}
\includecomment{TAKEIN}

\theoremstyle{definition}
\newtheorem{defn}{Definition}[section]
\newtheorem{remark}{Remark}[section]

\newtheorem{theo}{Theorem}[]
\newtheorem{lemma}{Lemma}[section]
\newtheorem{example}{Example}[section]
\newtheorem{proposition}{Proposition}[section]

\numberwithin{equation}{section}
\title[Stability in log-modified NLS]{
Stability in the log-modified nonlinear Schr\"odinger equation in dimension one
}
\author{Daniele Garrisi}
\address{School of Mathematical Sciences\\
University of Nottingham Ningbo China\\
199 Taikang East Road\\
Ningbo 315100\\
China}
\email{daniele.garrisi@nottingham.edu.cn}
\author{Yihan Liang}
\address{School of Mathematical Sciences\\
University of Nottingham Ningbo China\\
199 Taikang East Road\\
Ningbo 315100\\
China}
\email{smxyl1@nottingham.edu.cn}
\date{\today}
\begin{document}
\thispagestyle{empty}
\begin{abstract}
In the one-dimensional log-modified non-linear Schr\"odinger equation, 
there are finitely many normalised standing-waves, up to space translation and
phase multiplication. In the defocusing case, there is exactly one normalised standing-wave
having a prescribed mass and least energy. Both results are obtained by exploiting properties
of the mass function: the analyticity 
in the focusing case, and the monotonicity in the defocusing case.
\end{abstract}
\maketitle
\thispagestyle{empty}
\section*{Introduction}
In this work, we aim to study existence, orbital stability and uniqueness of
normalized standing-wave solutions to the log-modified non-linear Schr\"odinger equation
\begin{linenomath}
\begin{equation}
\label{eq.NLS}
i\partial_t\phi + \partial_{xx}\phi(t,x) - G'(\phi(t,x))\frac{\phi(t,x)}{|\phi(t,x)|} = 0,
\end{equation}
\end{linenomath}
where 
\begin{linenomath}
\begin{equation}
\label{eq.log-interaction}
G(s) = a|s|^p\log |s|^2.
\end{equation}
\end{linenomath}
According to \cite{CS23} and \cite{Mal19,PA16,TSK19}, the log-modified non-linear Schr\"odinger equation was introduced in the physics literature as a mean-field description of ultra-dilute quantum fluids in the spatial dimension \(n=2\). When \(p=2\),  standing-waves are known as Gaussons \cite{BM79}, and one of the proofs of their orbital stability was given by A.~Ardila in \cite{Ard16} in every dimension \(n\geq 2\).
Standing-waves are solutions to \eqref{eq.NLS} which can be written as
\begin{linenomath}
    \begin{equation}
        \label{eq.sw}
        \phi(t,x) = e^{-i\omega t}R(x) ,
    \end{equation}
\end{linenomath}
where \(\omega\in\mathbb{R}\) and \(R\in H^1(\mathbb{R}^n;\mathbb{R})\).
When \(p=2\) an essential ingredient to the proof of the orbital stability is the fact that every positive solution to 
\begin{linenomath}
    \begin{equation}
    \label{eq.log}
        \Delta R - G'(R) - \omega R = 0
    \end{equation}
\end{linenomath} 
can be written as \(R(x) = e^{\frac{1 - \omega}{a}} R_1(x + y)\), where \(R_1\) is the
unique \(H^1\) symmetric and positive solution to \eqref{eq.log} when \(\omega = 1\). 
The uniqueness of positive, radial and vanishing solution follows from 
\cite{McL93}. When \(2<p\) 
such a simple rescaling fails to exists and the problem of stability
and uniqueness of normalised waves is interesting even in the one-dimensional
case. For the particular case \(p=4\), R.~Carles and C.~Sparber \cite{CS23} proved the existence and stability of normalized standing-waves in dimension \(n=1\) via the method introduced in \cite{IK93}.

A normalized standing-wave is a solution to \eqref{eq.NLS} as in \eqref{eq.sw}
such that \(R\) is a minimum of the functional 
\begin{linenomath}
    \begin{gather}
    E\colon H^1(\mathbb{R};\mathbb{C})\to\mathbb{R}\\[0.5em]
        E(u) := \frac12\int_{-\infty}^{\infty} |u'(x)|^2 dx + \int_{-\infty}^{+\infty} G(u(x))dx
    \end{gather}
\end{linenomath}
on the constraint 
\begin{linenomath}
    \begin{gather}
S(\lambda) := \{u \in H^1(\mathbb{R},\mathbb{C})\mid   M(u)= \lambda\}\\[0.5em]
 M(u) := \int_{-\infty}^{+\infty} |u(x)|^2dx
    \end{gather}
\end{linenomath}
defined for every \(\lambda>0\). There are general and well-known
assumptions introduced in \cite{BBGM07} in dimension \(n\geq 3\)
and in \cite{GG17} in dimension one making the functional \(E\) well-defined
and bounded below on \(S(\lambda)\). Therefore, one can define
\begin{linenomath}
    \begin{equation}
        I\colon (0,+\infty)\to\mathbb{R},\quad I(\lambda) := \inf_{S(\lambda)} E.
    \end{equation}
\end{linenomath}
We define the two sets
\begin{linenomath}
    \begin{gather}
        \mathcal{G}_\lambda := \{u\in S(\lambda)\mid E(u) = 
        I(\lambda)\}\\[0.5em]
        \mathcal{G}_\lambda(u) := \{zu(\cdot +y)\mid (z,y)\in S^1\times\mathbb{R}\},
    \end{gather}
\end{linenomath}
where \(S^1\) is the set of complex numbers of modulus one. Before introducing
the main results, we define local well-posedness and stability.
\begin{defn}[Local well-posedness]
The non-linear Schr\"odinger equation is locally well-posed in \(H^1(\mathbb{R};\mathbb{C})\) if for every \(u_0\in H^1\), there exists
\(T>0\) and unique \(\phi\colon [0,T)\to H^1\) such that 
\begin{enumerate}[(i).]
\itemsep=0.3em
\item \(\phi(0,\cdot)=u_0\)
\item the map \([0,T)\ni t\to \phi(t,\cdot)\)  is 
\(C([0,T),H^1)\cap C^1([0,T),H^{-1})\)
\item for every \(t\in [0,T)\), \(\phi(t,\cdot)\) is a weak solution to 
\eqref{eq.NLS}.
\end{enumerate}
\end{defn}
When \eqref{eq.NLS} is globally well-posed, local solutions can be
extended to \((-\infty,+\infty)\), we set 
\(U_t(u) := \phi(t,\cdot)\) 
for every \(t\in\mathbb{R}\). On \(H^1(\mathbb{R};\mathbb{C})\), we consider the
inner product and the induced distance
\begin{linenomath}
    \begin{gather}
    \label{eq.inner}
        (u,v)_{H^1} := \text{Re}\int_{-\infty}^{+\infty} u(x)\overline{v(x)}dx + 
        \text{Re}\int_{-\infty}^{+\infty} u'(x)\overline{v'(x)}dx\\[0.5em]
        d(u,v) := \sqrt{(u - v,u - v)_{H^1}}.
        \end{gather}
\end{linenomath}
\begin{defn}[Stability]
A subset \(S\subseteq H^1\) is stable if 
\begin{enumerate}[(i).]
\itemsep=0.3em
\item \(U_t(S)\subseteq S\) for every \(t\geq 0\)
\item for every \(\varepsilon>0\), there
exists \(\delta>0\) such that \(d(u_0,S)<\delta\) implies 
\(d(U_t(u_0),S)<\varepsilon\) for every \(t\geq 0\).
\end{enumerate}
\end{defn}
\begin{defn}[Stability of standing-waves]
    The standing-wave \eqref{eq.sw} is stable if \(\mathcal{G}_\lambda(R)\)
    is stable.
\end{defn}
In literature, the set \(\mathcal{G}_\lambda\) is also called 
\textsl{ground state}, 
\cite{BBBM10}. The set \(\mathcal{G}_\lambda(R)\)
contains the solution of \eqref{eq.NLS} with initial value if \(u=R\),
and space translations. In fact, in \cite{CL82} it is shown that
if one omits space translations, the resulting set is never stable.
The existence of normalised standing-waves and the stability of \(\mathcal{G}_\lambda\) follows from well established techniques,
as the Concentration-Compactness Lemma of \cite{Lio84a,Lio84b} (in the version proposed
in \cite{BF14}) and the strict sub-additivity property of the function
\(I\), Lemma~\ref{lem.strict-subadditivity}. 

It is known that in the pure mass-critical power, 
\(G(s)=a|s|^{2 + \frac4n}\), \(E\) does not attain its infimum or is unbounded below \(S(\lambda)\) for every \(a\) in \(\mathbb{R}\). 
We refer to \cite[\S2]{GG23}
for a proof. In the log-modified pure power, \eqref{eq.log-interaction} however,
in the mass-critical case \(p = 2 + 4/n\) normalized waves exist
and minimizing sequences are compact up to translation as long as \(a>0\), while
they do not exist if \(a\leq 0\).
We prove this fact in dimension \(n=1\) in Lemma~\ref{lem.concentration-compactness}.
\begin{theo}
\label{thm.ground-state-log}
Suppose that \(2<p<6\) and \(a\neq 0\), or \(p=6\) and \(a>0\).
Then, \(\mathcal{G}_\lambda\) is non-empty and stable for every \(\lambda>0\).
\end{theo}
The main results of this paper focus on the stability of \(\mathcal{G}_\lambda(u)\).
Depending on whether \(a>0\) or \(a<0\), we rely on two different methods, both
based on properties of the mass function (Definition~
\ref{defn.mass-function})
\begin{linenomath}
    \begin{equation}
\lambda\colon\Omega\ni\omega\to 
\|R_\omega\|_{L^2(\mathbb{R})}^2\in (0,+\infty).
    \end{equation}
\end{linenomath}
The function \(R_\omega\) is the unique, positive and symmetric solution to \eqref{eq.log} in \(H^1\). We set
\begin{linenomath}
    \begin{equation}
    \label{eq.radial}
        \mathcal{G}_\lambda^{r,+} := \mathcal{G}_\lambda\cap H^1_{r,+}.
    \end{equation}
\end{linenomath}
\begin{theo}[Uniqueness of normalized standing-waves]
\label{thm.uniqueness-focusing}
Given \(a>0\) and \(2<p\leq 6\), the set 
\(\mathcal{G}_\lambda\cap H^1_{r,+}\) is a singleton
for every \(\lambda>0\).
\end{theo}
The theorem relies on a Euler differential inequality introduced in \cite{GG17}, as a sufficient condition to have \(\lambda'(\omega)> 0\)
for every \(\omega\in\Omega\).
If \(a<0\), the inequality is not satisfied nor we have a conclusion on the uniqueness. However, we can show that \(\mathcal{G}_\lambda^{r,+}\)
is a \textsl{finite set} under the following assumptions:

\hypertarget{G0}{(G0).}
\(G\) is even, \(C^1([0,+\infty))\cap C^{\omega}((0,+\infty))\) and
\(G'\) is locally Lipschitz.

\hypertarget{G1}{(G1).}
\label{G1} There exists \(s_0>0\) such that \(G(s_0)<0\).

\hypertarget{G2a}{(G2a).}
There exist \(2<p<q\) and \(C>0\) such that 
\(|G'(s)|\leq C(|s|^{p-1} + |s|^{q-1})\) for every \(s\).

\hypertarget{G2b}{(G2b).}
There exist \(2 < p_* < 6\) and \(C>0\) such that 
and \(G(s)\geq - C|s|^{p_*}\) for every \(s\geq s_*\).

From (G1), (G2a) and (G2b) and the local Lipschitz property of \(G'\), it follows
the stability of \(\mathcal{G}_\lambda\), \cite[Theorem~1.1]{GG17}. In this paper, we include the analyticity of \(G\) on \((0,+\infty)\). Consequently, the mass
function \(\lambda\) is analytic, Lemma~\ref{lemma.lambda_analytic}.
\begin{theo}
\label{thm.finite}
Suppose that \(G\) satisfies (\hyperlink{G0}{G0}), 
(\hyperlink{G1}{G1}),
(\hyperlink{G2a}{G2a}) and (\hyperlink{G2b}{G2b}). Then, \(\mathcal{G}_\lambda^{r,+}\) is a finite set.
\end{theo}
In \cite{Gar12} we remarked that when \(\mathcal{G}_\lambda^{r,+}\) is
finite, every normalized standing-wave is stable. Consequently, for a large class of non-linearities, normalized standing-waves are stable, including 
\eqref{eq.log-interaction} when \(a<0\).
\begin{theo}
\label{thm.stability-wave}
Suppose that \(G\) satisfies (\hyperlink{G0}{G0}), 
(\hyperlink{G1}{G1}),
(\hyperlink{G2a}{G2a}) and 
(\hyperlink{G2b}{G2b}). Then,
\(\mathcal{G}_\lambda (u)\) is stable for every \(u\in\mathcal{G}_\lambda\).
\end{theo}
This theorem applies, in particular, to a combination of an arbitrarily high 
number of powers, as in \cite{Mae08} and \cite{GG19}. In \cite{GG19}
uniqueness of normalized waves, and therefore stability, has been obtained 
for a combination of three powers, under some restrictions on the coefficients.
In Proposition~\ref{prop.combined-powers}, we prove that for the non-linearity
\begin{linenomath}
    \begin{gather}
    \label{eq.combined}
        G(s) = \sum_{i=1}^m \alpha_i |s|^{p_i},\quad 2 < p_i < p_{i+1}\\
        \alpha_i<0\Rightarrow p_i<6\ \forall i
    \end{gather}
\end{linenomath}
\(\mathcal{G}_\lambda(u)\) is stable for every \(u\in\mathcal{G}_\lambda\).
Finally, we describe the asymptotic behaviour of \(\lambda\) as \(\omega\to 0\). In particular, when
\(G\) is as in \eqref{eq.log-interaction} or
\eqref{eq.combined}, \(\lambda(\omega) = o(1)\) as per
Theorem~\ref{thm.log-asymptotic} and Proposition~\ref{prop.combined-powers},
respectively.
\section{Existence of minima and compactness up to translation}
In this section, we prove that minimizing sequences of \(E\) constrained
to \(S(\lambda)\) are compact up to translation in all the cases
where normalized standing-waves exist, namely \(2<p<6\) for
every \(a\neq 0\), and \(p=6\) for \(a>0\), while they do not exist
in all the other cases. The case \(p=2\) is omitted as \(E\) is not even
well-defined in \(H^1\). Instead, we refer to \cite{Ard16} for the stability
of standing-waves (Gaussons) when \(p=2\). In the remainder of this
section we will use the notations
\begin{linenomath}
    \begin{equation}
        D(u) := \frac12 \int^{+\infty}_{-\infty}  |u'(x)|^2 ,
        \quad P(u) := \int^{+\infty}_{-\infty}  G(|u(x)|)dx
    \end{equation}
\end{linenomath}
for every \(u\in H^1(\mathbb{R};\mathbb{C})\).
\begin{proposition}
\label{prop.smooth}
    \(E\) and \(M\) are \(C^2(H^1(\mathbb{R};\mathbb{C}),\mathbb{R})\).
\end{proposition}
\begin{proof}
Since \(p>2\), then \(G\in C^2(\mathbb{R})\) and 
\begin{linenomath}
    \begin{equation}
    \label{eq.second-derivative}
        G''(s) = as^{p-2} (p(p-1)\log s^2 + 4p - 2)
    \end{equation}
\end{linenomath}
for every \(s>0\), and \(G''(0) = 0\). To prove the regularity of \(E\), we show that 
\(G''\) satisfies a double-power estimate. Firstly, we consider the case 
\(2 < p < 6\). Let \(\varepsilon_0>0\) be
such that \(0 < p - 2 - \varepsilon_0 < p - 2 + \varepsilon_0 < 4\).
\begin{linenomath}
    \begin{equation}
    \begin{split}
    |a||s|^{p - 2 - \varepsilon_0} |s|^{\varepsilon_0}\log |s|^2&\leq 
    \||s|^{\varepsilon_0}\log|s|^2\|_{L^{\infty}(0,1)} |s|^{p - 2 - \varepsilon_0}\\
    &=\frac{2|a|}{e \varepsilon_0}|s|^{p - 2 - \varepsilon_0}
    \end{split}
    \end{equation}
\end{linenomath}
for every \(|s|\leq 1\). If \(|s|\geq 1\), we rely on the inequality 
\begin{linenomath}
    \begin{equation}
|a|p(p-1)|s|^{p - 2} \log|s|^2\leq \frac{2|a|p(p-1)}{e\varepsilon_0} |s|^{p - 2 + \varepsilon_0}.
\end{equation}
\end{linenomath}
Since 
\begin{linenomath}
    \begin{equation}
(4p-2)|s|^{p - 2}\leq (4p-2)(|s|^{p - 2 - \varepsilon_0} + |s|^{p - 2 + \varepsilon_0}), 
    \end{equation}
\end{linenomath}
combining estimates on \(|s|\leq 1\) and \(|s|\geq 1\), we obtain       
\begin{linenomath}
    \begin{equation}
    \begin{split}
    \label{eq.nemytskii}
        |G''(s)|\leq 
         c_1 (\varepsilon_0,a,p)(|s|^{p - 2 + \varepsilon_0} + |s|^{p - 2 + \varepsilon_0}).
        \end{split}
    \end{equation}
\end{linenomath}
If \(p=6\), the same estimates are satisfied, except that \(p + \varepsilon_0>6\).
Therefore, from \cite[Proposition~7]{GG17}, \(P\in C^2(H^1,\mathbb{R})\). As for \(M\), 
one can apply the same proposition when \(G(s) = s^2\). 
The kinetic part is also two-times continuously differentiable as 
\(D(u) = \|u\|_{H^1}^2 - M(u)\).
\end{proof}
\begin{proposition}
\label{prop.coercive}
Suppose that \(2 < p \leq 6\). 
If \(p = 6\) and \(a<0\), then \(E\) is not bounded below for every \(\lambda>0\). Otherwise, 
\(E\) is bounded below and coercive if
\(2 < p < 6\) and \(a\neq 0\) or \(p=6\) and \(a>0\). Moreover,
\(I(\lambda)<0\) for every \(\lambda>0\).
\end{proposition}
\begin{proof}
In the critical case, we can show that the functional \(E\) is not
bounded below using a conformal rescaling, just as it can be done
for pure-power nonlinearities, for instance, at the end of the
proof of \cite[Lemma~3.1]{GG17}. Given \(u_0\in S(\lambda)\) we set
\(u_R(x) := R^{\frac12}u_0 (Rx)\).
From the variable change \(Rx=y\) we obtain
\begin{linenomath}
    \begin{equation}
        \begin{split}
\|u_R\|_{L^2}^2 &= \int^{+\infty}_{-\infty}  |u_R(x)|^2 dx 
= \int^{+\infty}_{-\infty}  R |u_0(Rx)|^2 dx \\
&= \int^{+\infty}_{-\infty}  R|u_0(y)|^2 R^{-1}dy = \int^{+\infty}_{-\infty}|u_0(y)|^2 dy=\|u_0\|_{L^2}^2 = \lambda.
\end{split}
\end{equation}
\end{linenomath}
Thus, \(u_R \in S(\lambda)\) for all \(R>0\). Also,
\begin{linenomath}
    \begin{equation}
        \begin{split}
        D(u_R) &= \frac{1}{2} \int^{+\infty}_{-\infty} | u_R'(x)|^2 dx 
= \frac{1}{2} \int^{+\infty}_{-\infty}  \left|R^{\frac{1}{2}}  u_0'(Rx)R\right|^2 dx \\
&= \frac{1}{2} \int^{+\infty}_{-\infty}R^3  |u_0'(Rx)|^2 dx =\frac{1}{2}\int^{+\infty}_{-\infty}R^3R^{-1}|u_0'(y)|^2 dy \\[0.5em]
&=R^2 D(u_0)
        \end{split}
    \end{equation}
\end{linenomath}
and
\begin{linenomath}
    \begin{equation}
    \begin{split}
        P(u_R) &= a \int^{+\infty}_{-\infty} |u_R(x)|^p\log |u_R(x)|^2 dx\\ 
&= a \int^{+\infty}_{-\infty} R^{\frac{p}2} |u_0(Rx)|^p \log |R^{\frac12}u_R(Rx)|^2 dx \\
&= a R^{\frac{p}2-1} \int^{+\infty}_{-\infty} |u_0(y)|^p 
\log |R^{\frac12} u_0(y)|^2 dy \\[0.8em]
&= R^{\frac{p}{2}-1} P(u_0) + aR^{\frac{p}2-1}\log R \int^{+\infty}_{-\infty} |u_0|^p dy.
\end{split}
    \end{equation}
\end{linenomath}
Thus, 
\begin{equation}
\label{eq.coercive.1}
\begin{aligned}
E(u_R) &= D(u_R) + P(u_R) \\
&= R^2 D(u_0) + R^2 P(u_0) + aR^{\frac{p}2-1}\log R
\int^{+\infty}_{-\infty} |u_0|^pdy\\
&= R^2 (D(u_0) + P(u_0)) + aR^{\frac{p}2-1}\log R
\int^{+\infty}_{-\infty} |u_0|^p dy \\
&= R^2 E(u_0) + aR^{\frac{p}2 - 1}\log R
\int^{+\infty}_{-\infty} |u_0|^pdy \\
&= R^2(E(u_0) + a\|u_0\|_p^p    R^{\frac{p}2 - 3}\log R).
\end{aligned}
\end{equation}
If \(a<0\) and \(p=6\), then the energy \(E_R(u_0)\) is unbounded below when 
\(R\) diverges to \(+\infty\). 

\paragraph*{\(p<6\) or \(a>0\)} In the sub-critical case \(2 < p < 6\), lower bounds and coercivity
can be obtained from the Gagliardo-Nirenberg inequality:
for every \(2 < p < 6\),
\begin{linenomath}
\begin{equation}
\label{eq.coercive.gn}
\|u\|_{L^p(\mathbb{R})}^p\leq 2 \|u\|_{L^2(\mathbb R)}^{\frac{p+2}{2}}
\|u'\|_{L^2(\mathbb R)}^{\frac{p-2}{2}}.
\end{equation}
\end{linenomath}
We refer to \cite{DEL14} for the optimal constant. A double integration of
the inequality \eqref{eq.nemytskii} gives
\begin{linenomath}
\begin{equation}
    \label{eq.coercive.5}
|G(s)|\leq c_1(|s|^{p - \varepsilon_0} + 
|s|^{p + \varepsilon_0}),
\end{equation}
\end{linenomath}
Therefore, for every \(u\in S(\lambda)\)
\begin{linenomath}
\begin{equation}
    \begin{split}
E(u) &= \frac12\int_{-\infty}^{+\infty} |u'|^2dx + a\int_{-\infty}^{+\infty}
|u|^p \log |u|^2dx\\
&\geq \frac12\int_{-\infty}^{+\infty} |u'|^2dx - 
c_1\int_{-\infty}^{+\infty} |u|^{p - \varepsilon_0}dx
- c_1\int_{-\infty}^{+\infty}|u|^{p + \varepsilon_0}dx.
\end{split}
\end{equation}
\end{linenomath}
From \eqref{eq.coercive.gn},
\begin{linenomath}
    \begin{equation}
    \label{eq.cc.8}
        \begin{split}
        E(u)
&\geq \frac12\int_{-\infty}^{+\infty} |u'|^2dx - 
2c_1 \|u\|_{L^2(\mathbb R)}^{\frac{p - \varepsilon_0 +2}{2}}
\|u'\|_{L^2(\mathbb R)}^{\frac{p - \varepsilon_0-2}{2}} - 
2c_1\|u\|_{L^2(\mathbb R)}^{\frac{p + \varepsilon_0 +2}{2}}
\|u'\|_{L^2(\mathbb R)}^{\frac{p + \varepsilon_0-2}{2}}\\
&=\frac12 D - 2c_1\lambda^{\frac{p - \varepsilon_0+2}{4}} D^{\frac{p - \varepsilon_0-2}{4}} - 2c_1\lambda^{\frac{p + \varepsilon_0+2}{4}} D^{\frac{p + \varepsilon_0-2}{4}}\\
&=: \frac12 D - 2c_1 \lambda^{\gamma_1 + 1}D^{\gamma_1} - 
2c_1\lambda^{\gamma_2 + 1}D^{\gamma_2},
\end{split}
\end{equation}
\end{linenomath}
where
\begin{linenomath}
\begin{equation}
\gamma_1 := \frac{p - \varepsilon_0+2}{4},\quad
\gamma_2 := \frac{p + \varepsilon_0-2}{4}.
\end{equation}
\end{linenomath}
Since \(p+\varepsilon_0<6\), there holds \(0 < \gamma_i < 1\) for \(i=1,2\). Therefore, \(E\) is bounded below. Moreover, if \(E\) is bounded above,
then \(D\) is bounded above. Since \(\|u\|_2^2 = \lambda\), 
the \(H^1\) norm of \(u\) is bounded above. Therefore, \(E\) is bounded below and coercive.

If \(p=6\) and \(a>0\), it is convenient to integrate in a neighbourhood of the origin and in a neighbourhood of infinity. From \eqref{eq.coercive.5} (\(p=6\)),
and \eqref{eq.coercive.gn},
\begin{linenomath}
    \begin{equation}
    \label{eq.coercive.8}
\begin{split}
E(u) &= \frac12\int_{-\infty}^{+\infty} |u'|^2dx + 
a\int_{|u|\leq 1} |u|^6\log |u|^2 dx + 
a\int_{|u|> 1} |u|^6\log |u|^2 dx\\
&\geq  \frac12\int_{-\infty}^{+\infty} |u'|^2dx  - \frac{2a}{e\varepsilon_0}
\int_{|u|\leq 1} |u|^{6 - \varepsilon_0}dx
+ a\int_{|u|> 1} |u|^6\log |u|^2 dx\\
&\geq \frac12\int_{-\infty}^{+\infty} |u'|^2dx  - \frac{2a}{e\varepsilon_0}
\int_{-\infty}^{+\infty} |u|^{6 - \varepsilon_0}dx
+ a\int_{|u|> 1} |u|^6\log |u|^2 dx\\
&\geq \frac12\int_{-\infty}^{+\infty} |u'|^2dx  - 
c_2\lambda^{\frac{8 - \varepsilon_0}{4}} D^{\frac{4 - \varepsilon_0}{4}}\\
&= \frac12 D - c_2 \lambda^{\gamma_3+1} D^{\gamma_3},
\end{split}
\end{equation}
\end{linenomath}
where \(\gamma_3=(8-\varepsilon_0)/4\) and \(c_2 = 2a/(e\varepsilon_0)\).
The last inequality follows from the fact that \(G(s)>0\) when \(s>1\).
Since \(\gamma_3<1\), \(E\) is bounded below and coercive. 
\paragraph{\(I(\lambda)<0\)}
From \eqref{eq.coercive.1}, if \(a>0\) and \(2 < p \leq 6\), 
\(E(u_R)<0\) when \(R\to 0\). If \(a<0\), and \(2 < p < 6\), 
\eqref{eq.coercive.1} does not provide enough information to conclude
that \(I(\lambda)<0\) as when \(R\to 0\), the dominating term is positive.

Therefore, we construct test functions in a similar fashion
as (2.2) and (2.3) of \cite[Lemma~5]{BBGM07} with appropriate
modifications: in fact, since we are in dimension one, the measure
of the annulus \(B(0,r+1) - \overline{B}(0,r)\) has the
same order of the measure of the ball \(B(0,r)\). Given \(s>0\) and \(r>0\), we set
\begin{linenomath}
    \begin{equation}
v_{r,s}(x) := 
\begin{cases}
    s & \text{ if } |x|\leq r\\[0.3em]
    -\frac{s}{r^\alpha}(|x| - r) + s & \text{ if } r\leq |x|\leq r + 
    r^\alpha\\[0.3em]
    0 & \text{ otherwise}.
\end{cases}
\end{equation}
\end{linenomath}
\(v_{r,s}\in H^1\) and
\begin{linenomath}
\begin{equation}
v_{r,s}'(x) = 
\begin{cases}
    0 & \text{ if } |x|\leq r\\[0.3em]
    -\frac{sx}{|x|r^\alpha} & \text{ if } r<|x|\leq r + r^\alpha\\[0.3em]
    0 & \text{ otherwise}.
\end{cases}
\end{equation}
\end{linenomath}
Therefore,
\begin{linenomath}
\begin{equation}
D(v_{r,s}) = \int_{-(r+r^\alpha)}^{-r} \frac{s^2}{r^{2\alpha}} dx + \int_{r}^{r+r^\alpha} \frac{s^2}{r^{2\alpha}}dx = \frac{2s^2}{r^{\alpha}}.
\end{equation}
\end{linenomath}
Since \(G(0) = 0\) and \(0\leq v_{r,s}(x)\leq s\) on \((-r-r^\alpha,-r)\cup(r,r+r^\alpha)\),
we have
\begin{linenomath}
\begin{equation}
\begin{split}
P(v_{r,s}) =&\, 2r G(s) + \int_{r\leq |x|\leq r+r^{\alpha}} 
G(v_{r,s}(x))dx\\[0.6em]
\leq&\, 2rG(s) + 2 r^{\alpha}\|G\|_{L^\infty(0,s)}.
\end{split}
\end{equation}
\end{linenomath}
Since \(a<0\), if \(s\geq 1\) is large enough, \(G\) is decreasing and 
\(G(s)\leq-\|G\|_{L^\infty(0,1)}\). For such values of \(s\) we have
\(\|G\|_{L^\infty(0,s)} = |G(s)|\). Therefore,
\[
\begin{split}
E(v_{r,s})&\leq\frac{s^2}{r^\alpha} + 2rG(s) + 2r^\alpha\|G\|_{L^\infty(0,s)}\\[0.4em]
&= \frac{s^2}{r^\alpha} + 2rG(s) + 2r^\alpha |G(s)|.
\end{split}
\]
The function \(v_{r,s}\) is in \(S(\lambda)\) if and only if 
\begin{equation}
\label{eq.coercive.7}
\begin{split}
&\,2rs^2 + \int_{r\leq |x|\leq r+r^\alpha} |v_{r,s}(x)|^2 dx \\
=&\,2rs^2 + \int_{-(r+r^\alpha)}^{-r} |v_{r,s}(x)|^2dx + 
\int_{r}^{r+r^\alpha} |v_{r,s}(x)|^2dx\\[0.4em]
=&\,2rs^2 + 2\int_0^s y^2 \frac{r^\alpha}{s} dy
= 2rs^2 + \frac{2r^{\alpha}s^2}{3} = \lambda. 
\end{split}
\end{equation}
The second equality follows from the variable change \(v_{r,s}(x)=y\).
For \(E(v_{r,s})\) to achieve negative values, \(s\) must be chosen
large enough, so that \(G(s)<0\). From \eqref{eq.coercive.7},
if \(s\) diverges, \(r\) converges to zero. Therefore, 
\(
2r^\alpha |G(s)| = o(2r|G(s)|)
\)
as \(r\to 0\) if \(\alpha > 1\). To obtain
\(
s^2 r^{-\alpha} = o(2r|G(s)|)
\)
we can refine the choice of \(\alpha\). Since \(r^\alpha=o(r)\),
there holds \(r\sim \lambda/s^2\). Therefore, 
\begin{linenomath}
\begin{equation}
\begin{split}
\frac{s^2 r^{-\alpha}}{2r|G(s)|} &= \frac{s^2}{2r^{\alpha+1}|G(s)|}\sim
\frac{\lambda^{-\alpha - 1}s^2}{(1/s^2)^{\alpha+1}|G(s)|} \\
&= 
\frac{\lambda^{-\alpha - 1}s^{2\alpha+4}}{|G(s)|} = 
\frac{\lambda^{-\alpha - 1}s^{2\alpha+4}}{|a|s^p \log |s|^2}
= 
\frac{\lambda^{-\alpha - 1}}{|a|\log|s|^2}\cdot s^{2\alpha + 4 - p}.
\end{split}
\end{equation}
\end{linenomath}
If we choose 
\[
\alpha>\max\left\{1,\frac{p-4}{2}\right\} = 1
\]
for \(s\) large enough and \(r\) satisfying \eqref{eq.coercive.7}, 
\(E(v_{r,s})<0\). For the sake of completeness we cover the case \(a=0\).
In \eqref{eq.coercive.1}, \(E(u_R) = R^2 D\geq 0\). Therefore,
\(E\) is bounded below and it is coercive. However, a minimum does not
exist as \(E(u_0) = 0\) implies \(D=0\). Therefore, \(u_0\) should be
constant with mass \(\lambda>0\), which contradicts \(u_0\in L^2\).
\end{proof}
\begin{remark}
In \cite[Theorem~1.1]{GG17}, it has been proved that 
\(E\) is coercive, bounded below with negative infimum 
on \(S(\lambda)\) provided \(\lambda\) is large enough. In \cite[Corollary~3]{BBGM07}, a global result on every constrain has been obtain 
in dimension \(n\geq 3\), provided the assumption (referred to in the paper as \(F_2\)) is satisfied, \(G(s)\leq -s^{6 - \varepsilon_0}\) for small \(s\).
However, when \(a<0\) this inequality
is not satisfied. This is why a proof independent on
\cite[Theorem~1.1]{GG17} and \cite[Lemma~5]{BBGM07} was necessary.
\end{remark}
For the remainder of this section, we are going to assume that 
\begin{linenomath}
    \begin{equation}
(a,p)\in (0,+\infty)\times (2,6]\cup (-\infty,0)\times (2,6).
    \end{equation}
\end{linenomath}
\begin{proposition}
\label{prop.lower-semicontinuous}
\(E\) is weakly lower semi-continuous on \(S(\lambda)\).
\end{proposition}
\begin{proof}
Since \(D\colon H^1\to\mathbb{R}\) is weakly lower semi-continuous,
it is enough to prove that \(P\colon H^1\to\mathbb{R}\)
is weakly lower semi-continuous. Let \((u_n)_{n\geq 1}\) be sequence converging weakly
in \(H^1\) to \(u\in S(\lambda)\). Therefore, \(u_n\)
converges pointwise almost everywhere to \(u\) and \(u_n\rightharpoonup u\) in \(L^2\). Since \(\|u\|_{L^2} = \|u_n\|_{L^2}\)
for every \(n\geq 1\), \(u_n\to u\) in \(L^2\).
From \cite[Theorem~4.9]{Bre10}, there exists a subsequence of \((u_{n_k})_{k\geq 1}\)
and \(g\in L^2\) such that \(|u_{n_k}(x)|\leq g(x)\).
Since \((u_n)_{n\geq 1}\) is bounded in \(H^1\), 
it is also bounded in \(L^\infty\). From \eqref{eq.coercive.5},
\begin{linenomath}
    \begin{equation}
        |G(u_n(x))|\leq c_1 |g(x)|^2 
        \Big(\sup_{n\geq 1}\|u_n\|_{L^\infty}^{p - \varepsilon_0 - 2}
        + \sup_{n\geq 1}\|u_n\|_{L^{\infty}}^{p + \varepsilon_0 - 2}
        \Big).
    \end{equation}
\end{linenomath}
Therefore, \(|G\circ u_n|\) converges to \(|G\circ u|\) in \(L^1\), 
that is \(P(u_n)\) converges to \(P(u)\) and
\(\liminf_{n\to\infty} E(u_n)\geq E(u)\).
\end{proof}
\begin{defn}
    Given \(\lambda>0\), we set
    \begin{linenomath}
        \begin{equation}
            d(\lambda) := \lim_{\varepsilon\to 0}
            \inf_{S(\lambda)\cap \{E\leq I(\lambda) + \varepsilon\}} D.
        \end{equation}
    \end{linenomath}
\end{defn}
\begin{lemma}
\label{lem.dlambda}
For every \(\lambda>0\), there holds \(d(\lambda)>0\).
\end{lemma}
\begin{proof}
On the contrary, suppose that the limit converges to zero.
For every \(\varepsilon=1/n\), there exists \(u_n\in S(\lambda)\)
such that \((u_n)_{n\geq 1}\) is a minimizing sequence and
\(D(u_n)\to 0\). If \(2 < p < 6\), from \eqref{eq.cc.8} we obtain 
\begin{linenomath}
    \begin{equation}
\begin{split}
|P(u_n)|\leq 2c_1 \lambda^{\gamma_1+1}D(u_n)^{\gamma_1}
+ 2c_1\lambda^{\gamma_2 + 1}D(u_n)^{\gamma_2}.
\end{split}
\end{equation}
\end{linenomath}
If \(p=6\) and \(a>0\), a similar conclusion can be made
from \eqref{eq.coercive.8}. Therefore, \(I(\lambda) = 
\lim_{n\to\infty} P(u_n) + D(u_n) = 0\) contradicts 
Proposition~\ref{prop.coercive}.
\end{proof}
\begin{lemma}
The function 
\begin{linenomath}
    \begin{equation}
    (0,+\infty)\ni\lambda\to -\frac{I(\lambda)}{\lambda}\in (0,+\infty)
    \end{equation}
\end{linenomath}
is strictly increasing.
\end{lemma}
\begin{proof}
Let \((u_n)_{n\geq 1}\) be a minimizing sequence 
of \(E\) in \(S(\lambda)\).
Given \(\vartheta>1\),
\begin{linenomath}
    \begin{equation}
        u_{n,\vartheta}(x) := u_n(\vartheta^{-1}x).
    \end{equation}
\end{linenomath}
Since \(u_n\in S(\vartheta\lambda)\),
\begin{linenomath}
    \begin{equation}
    \begin{split}
        I(\vartheta\lambda)\leq E(u_{n,\vartheta}) &= 
        \frac12\vartheta^{-1}D(u_n) + \vartheta P(u_n)  \\
        &= \frac12 (\vartheta^{-1} - \vartheta)D(u_n) + \vartheta \left(\frac12 D(u_n) + P(u_n)\right)\\
        &= \frac12 (\vartheta^{-1} - \vartheta)D(u_n) + \vartheta I(\lambda) + o(1)\\
        &\leq \frac12 (\vartheta^{-1} - \vartheta) d(\lambda) + \vartheta I(\lambda) + o(1),
        \end{split}
    \end{equation}
\end{linenomath}
whence
\begin{linenomath}
    \begin{equation}
    \label{eq.strict-subadditivity.1}
    \frac{2\lambda\left(\frac{I(\lambda)}{\lambda} - \frac{I(\vartheta\lambda)}{\vartheta\lambda}\right)}{1 - \vartheta^{-2}} = 
        2\cdot\frac{I(\vartheta\lambda)-\vartheta I(\lambda)}{\vartheta^{-1} - \vartheta}\geq d(\lambda)
    \end{equation}
\end{linenomath}
which concludes the proof. 
\end{proof}
\begin{lemma}
\label{lem.strict-subadditivity}
For every \(2 < p\leq 6\) with \(a>0\), and \(2 < p < 6\)
with \(a<0\), there holds
\begin{linenomath}
    \begin{equation}
I(\lambda_1 + \lambda_2)<I(\lambda_1) + I(\lambda_2) - 2\lambda_1\lambda_2\kappa(G) + \lambda_1 I(\lambda_2) + \lambda_2 I(\lambda_1).
\end{equation}
\end{linenomath}
for every \(\lambda_1,\lambda_2>0\).
\end{lemma}
\begin{proof}
We apply \eqref{eq.strict-subadditivity.1} to \(\vartheta=\lambda_1^{-1}(\lambda_1 + \lambda_2)\)
and \(\lambda_2^{-1}(\lambda_1 + \lambda_2)\):
\begin{linenomath}
    \begin{equation}
    \begin{split}
        I(\lambda_1 + \lambda_2)&\leq\frac12\left(\frac{\lambda_1}{\lambda_1 + \lambda_2} - 
        \frac{\lambda_1 + \lambda_2}{\lambda_1}\right) d(\lambda_1) + \frac{\lambda_1 + \lambda_2}{\lambda_1}I(\lambda_1)\\[0.5em]
        I(\lambda_1 + \lambda_2)&\leq
        \frac12\left(\frac{\lambda_2}{\lambda_1 + \lambda_2} - 
        \frac{\lambda_1 + \lambda_2}{\lambda_2}\right)d(\lambda_2)
        + \frac{\lambda_1 + \lambda_2}{\lambda_1}I(\lambda_2).
    \end{split}
    \end{equation}
\end{linenomath}
Multiplying the first and the second inequality by \(\lambda_1\) and \(\lambda_2\), respectively and taking the sum, we obtain
\begin{linenomath}
    \begin{equation}
    \begin{split}
        (\lambda_1 + \lambda_2)I(\lambda_1 + \lambda_2)\leq &\,
        \frac12\left(\lambda_1 - 
        \frac{(\lambda_1 + \lambda_2)^2}{\lambda_1}\right)d(\lambda_1) + 
        (\lambda_1 + \lambda_2)I(\lambda_1)\\
        + &\,\frac12 \left(\lambda_2 - 
        \frac{(\lambda_1 + \lambda_2)^2}{\lambda_2}\right) d(\lambda_2) +
        (\lambda_1 + \lambda_2)I(\lambda_2).
        \end{split}
    \end{equation}
\end{linenomath}
Then
\begin{linenomath}
    \begin{equation}
    \begin{split}
        I(\lambda_1) + I(\lambda_2) - I(\lambda_1 + \lambda_2) &> 
         \frac12
        \frac{\lambda_2^2 + 2\lambda_1\lambda_2}{\lambda_1(\lambda_1 + \lambda_2)} \cdot d(\lambda_1) + 
        \frac12
        \frac{\lambda_1^2 + 2\lambda_1\lambda_2}{\lambda_2(\lambda_1 + \lambda_2)} \cdot d(\lambda_2)\\
        &= \frac{\lambda_2}{2\lambda_1} d(\lambda_1) + 
        \frac{\lambda_1}{2\lambda_2}d(\lambda_2).
        \end{split}
    \end{equation}
\end{linenomath}
The last equality follows from 
\begin{linenomath}
    \begin{equation}
         \frac{\lambda_2^2 + 2\lambda_1\lambda_2}{\lambda_1(\lambda_1 + \lambda_2)}
         >\frac{\lambda_2}{\lambda_1},\quad
         \frac{\lambda_1^2 + 2\lambda_1\lambda_2}{\lambda_2(\lambda_1 + \lambda_2)}
         >\frac{\lambda_1}{\lambda_2}.
    \end{equation}
\end{linenomath}
Therefore,
\begin{linenomath}
    \begin{equation}
    \label{lem.strict-subadditivity.2}
         I(\lambda_1) + I(\lambda_2) - I(\lambda_1 + \lambda_2) > 
         \frac12\left(\frac{\lambda_2 d(\lambda_1)}{\lambda_1} + 
         \frac{\lambda_1 d(\lambda_2)}{\lambda_2}\right)
    \end{equation}
\end{linenomath}
From Lemma~\ref{lem.dlambda}, the right-hand side of 
\eqref{lem.strict-subadditivity.2} is positive.
\end{proof}
\begin{lemma}
    \label{lem.concentration-compactness}
Suppose that one of the
two assumptions are satisfied:
\begin{enumerate}[(i).]
\itemsep=0.3em
\item \(2 < p < 6\) and \(a\neq 0\)
\item \(p=6\) and \(a>0\).
\end{enumerate}
For every \(\lambda>0\) and every minimizing sequence \((u_n)_{n\geq 1}\) of \(E\)
on \(S(\lambda)\), there exists \((u_{n_k})_{k\geq 1}\) 
subsequence and \((y_k)_{k\geq 1}\)
real sequence such that 
\begin{linenomath}
    \begin{equation}
        u_{n_k}(\cdot + y_k)\to u\text{ in } H^1
    \end{equation}
\end{linenomath}
for some \(u\in H^1(\mathbb{R};\mathbb{C})\).
\end{lemma}
\begin{proof}
Let \((u_n)_{n\geq 1}\) be a minimizing sequence
of \(E\) on \(S(\lambda)\). 
From (i) of Proposition~\ref{prop.coercive}, the sequence is bounded in \(H^1\). We apply one of the possible forms of the Concentration-Compactness Lemma, \cite{Lio84a,Lio84b}, specifically the one formulated in \cite{BF14}: up to extract a subsequence,
there exists \((y_n)_{n\geq 1}\subseteq\R\) and \(v\in S(\lambda)\)
such that \(u_n(\cdot + y_n)\to u\) in \(H^1\). On the contrary,
we have the two cases:

First case (Vanishing). For every subsequence \((u_{n_k})_{k\geq 1}\)
and sequence \((y_k)_{k\geq 1}\) such that \(u_{n_k}(\cdot + y_k)\rightharpoonup u\), there holds \(u = 0\). 
From \cite[Lemma~I.1]{Lio84b} or \cite{GG17}, the sequence
\((\|u_n\|_{L^p})_{n\geq 1}\) converges to zero for every \(p>2\).
From \eqref{eq.cc.8} (if \(2 < p < 6\)) or \eqref{eq.coercive.8} 
(if \(p=6\) and \(a>0\)), \(P(u_n)\to 0\), implying \(I(\lambda)\geq 0\),
contradicting Proposition~\ref{prop.coercive}.

Second case (Dichotomy). There exists a subsequence
 \((u_{n_k})_{k\geq 1}\)
and a sequence \((y_k)_{k\geq 1}\) such that \(u_{n_k}(\cdot + y_k)\rightharpoonup u\), where \(0 < \|u\|_{H^1} < \liminf_{n\to\infty} \|u_{n_k}\|_{H^1}\).
We can show that \(0 < \|u\|_{L^2}^2 < \lambda\). In fact,
if \(\|u\|_{L^2}^2 = \lambda\), by Proposition~\ref{prop.lower-semicontinuous}, 
\begin{linenomath}
    \begin{equation}
    \begin{split}
         I(\lambda) &= \liminf_{n\to\infty} E(u_n)\\
        &= \frac12\liminf_{n\to\infty} \bigg(\frac12 D(u_n) + 
        P(u_n)\bigg)\\[0.5em]
        &=\frac12\liminf_{n\to\infty} D(u_n) + 
        \lim_{n\to\infty} P(u_n)\\[0.5em]
        &\geq \frac12 D(u) + P(u) = E(u) = I(\lambda).
        \end{split}
    \end{equation}
\end{linenomath}
Therefore, the fourth term equals the fifth term,
implying 
\begin{linenomath}
    \begin{equation}
        \liminf_{n\to\infty} D(u_n) = D(u).
    \end{equation}
\end{linenomath}
Then, \(\|u_n\|_{H^1}^2 = D(u_n) + \|u_n\|_{L^2}^2\to D(u) + \|u\|_{L^2}^2 = \|u\|_{H^1}^2\), contradicting the dichotomic assumption. We now set \(\lambda_1 := \|u\|_{L^2}^2\),
\(\lambda_2 := \lambda - \lambda_1\) and 
\(v_k := u_{n_k}(\cdot + y_k) - u\). Therefore, 
\begin{linenomath}
    \begin{equation}
        \begin{split}
            I(\lambda) &= E(u_{n_k}) +o(1) = E(u_{n_k}(\cdot + y_k)) + o(1)\\[0.5em]
            &= E(u_{n_k}(\cdot + y_k) - u) + E(u) + o(1)\\[0.5em]
            &= E(v_k) + E(u) + o(1)\\
            &= E\left(\frac{\lambda_2^{\frac12} v_k}{\|v_k\|_2}\right) + E(u)+ o(1)\\[0.4em]
            &\geq I(\lambda_2) + I(\lambda_1) + o(1).
        \end{split}
    \end{equation}
\end{linenomath}
The second equality follows from \cite[Appendix~4]{BBGM07} or
\cite[Theorem~2]{BL83}. Taking the limit,
 we obtain a contradiction with Lemma~\ref{lem.strict-subadditivity}.
Therefore, \(v_{n_k}(\cdot + y_k)\to v\) in \(H^1\).
\end{proof}
\section{Uniqueness of normalized standing-waves when \(a>0\)}
We prove that if \(a>0\), then for every \(\lambda>0\)
there exists a unique \(R\in\mathcal{G}_\lambda\) which is 
positive, and symmetric-decreasing. We recall some 
fundamental facts about existence and
uniqueness of solutions to autonomous
non-linear elliptic PDEs in dimension one. They follow from \cite[Theorem~5]{BL83a} and related remarks. We set
\begin{linenomath}
    \begin{equation}
        V(s) := -\frac{2G(s)}{s^2}.
    \end{equation}
\end{linenomath}
\begin{proposition}
\label{prop.ber-lio-str}
Suppose that \(G\) is continuously differentiable and that 
\begin{linenomath}
    \begin{equation}
    \label{eq.ber-lio-str.1}
        \lim_{s\to 0} \frac{G(s)}{s^2} = \lim_{s\to 0}\frac{G'(s)}{s} =  0.
    \end{equation}
\end{linenomath}
For every \(\omega\in (0,\sup(V))\),
there exists a unique solution \(R_\omega\in H^1(\mathbb{R})\) to \eqref{eq.log} 
such that
\begin{enumerate}[(i).]
\itemsep=0.3em
\item \(R_\omega\) is positive and strictly symmetric-decreasing and \(H^1\)
\item \(R_\omega(0) = T(\omega) := \inf\{s > 0 \mid V(s) = \omega\}\)
\end{enumerate}
if and only if \(V'(T(\omega))>0\).
\end{proposition}
\begin{proof}
We set \(f(s) := - G'(s) - \omega s\) and \(F(s) := -G(s) - \omega s^2/2\).
From \cite[Theorem~5]{BL83a}, there exists a solution to \(u'' + f(u) = 0\) such that
\(\lim_{|x|\to\infty} u(x) = 0\) if and only if 
\(\zeta_0 := \inf\{s>0\mid F(s) = 0\}\) is positive and \(f(\zeta_0)>0\).
In the proof of \cite[Theorem~5]{BL83a}, this solution is constructed
with the initial value problem
\begin{linenomath}
    \begin{equation}
u''(x) + f(u(x)) = 0,\quad u(0) = \zeta_0,\quad u'(0) = 0.
    \end{equation}
\end{linenomath}
The condition \(F(\zeta_0) = 0\) is equivalent to 
\begin{linenomath}
    \begin{equation}
        -G(\zeta_0) - \frac{\omega\zeta_0^2}{2} = 0
    \end{equation}
\end{linenomath}
which implies \(V(\zeta_0) = \omega\). Since 
\begin{linenomath}
    \begin{equation}
    \label{prop.ber-lio-str.2}
    \begin{split}
        f(\zeta_0) &= -G'(\zeta_0) - \omega \zeta_0 = 
        -G'(\zeta_0) - V(\zeta_0)\zeta_0 \\[0.3em]
        &= -G'(\zeta_0) + \frac{2G(\zeta_0)}{\zeta_0^2}\cdot \zeta_0 
        = \frac{1}{\zeta_0^2}\bigg(-G'(\zeta_0)\zeta_0^2 + 2G(\zeta_0)\zeta_0\bigg) \\
        &= V'(\zeta_0),
        \end{split}
    \end{equation}
    \end{linenomath}
The condition \(f(\zeta_0)>0\) is equivalent to \(V'(\zeta_0)>0\). 

From the first limit in \eqref{eq.ber-lio-str.1}, \(V(0) = 0\). Since \(\omega>0\), there exists \(\zeta>0\) such that \(V(\zeta) = \omega\).
From \eqref{prop.ber-lio-str.2}, and the second limit in \eqref{eq.ber-lio-str.1}, 
\[
V'(s) = -G'(s) - \omega s = s\left(-\frac{G'(s) + \omega s}{s}\right)\sim\omega s.
\]
Since \(\omega>0\), \(V'>0\) on \((0,\varepsilon)\) for \(\varepsilon>0\) small
enough. Therefore, \(T(\omega)>0\). If \(V'(T(\omega))>0\) is positive, then
a unique positive, vanishing at infinity, symmetric solution to \eqref{eq.log} 
exists. We denote it \(R_\omega\).
To prove that \(R_\omega\in H^1\), we rely on \cite[Remark~6.3]{BL83a}:
if \(\lim_{s\to 0} f(s)/s =: -m < 0\), then \(R_\omega\) and \(R_\omega'\)
have exponential decay at infinity. In our case, 
\begin{linenomath}
    \begin{equation}
        \frac{f(s)}{s} = -\frac{G'(s) + \omega s}{s} = 
        -\frac{G'(s)}{s} - \omega.
    \end{equation}
\end{linenomath}
From \eqref{eq.ber-lio-str.1}, and \(\omega > 0\), assumptions of the quoted remark are satisfied.
\begin{example}
When \(V\) is invertible, \(T\) is just the inverse function of \(V\).
This is the case where \(G(s)=-a|s|^p\); in the double power case
\(G(s) = -a|s|^p + b|s|^q\) and \(G(s) = a|s|^p \log|s|^2\), 
\(T\) is the inverse function of \(V\) only in the interval \([0,s_2^*]\), where \(s_2^*\) is the unique local
maximum of \(V\).
\end{example}
\begin{defn}[Mass function]
\label{defn.mass-function}
We denote \(\Omega\) the set of \(\omega\in(0,\sup(V))\)
such that \(V'(T(\omega))>0\). We also denote
\begin{linenomath}
    \begin{equation}
        \lambda\colon\Omega\to (0,+\infty),\quad
        \lambda(\omega) := M(R_\omega).
    \end{equation}
\end{linenomath}
\end{defn}
\begin{remark}
\label{rem.domain-log}
If \(G\) is as in \eqref{eq.log} with \(a>0\), then the domain
of \(\lambda\) is given by \((0,\sup(V))\). Therefore 
\(\lambda\colon (0,\sup(V))\to (0,+\infty)\) is a well-defined
function. We set \(\sup(V) := \omega_2^*\).
\end{remark}
\end{proof}
\begin{proposition}
If \(2 < p \leq 6\), then 
\(G\) satisfies the Euler differential inequality
\begin{linenomath}
    \begin{equation}
    \label{eq.Euler}
        L(s) := 12G(s) - 7sG'(s) + s^2 G''(s)\geq 0
    \end{equation}
\end{linenomath}
for every \(s\in (0,s_2^*)\), where
\(s_2^* := \argmax(V)\).
\end{proposition}
\begin{proof}
If \(a>0\), then \(V(s) = -2a|s|^{p-2}\log |s|^2\). The function
is bounded above and the supremum is achieved by the
unique positive zero of \(V'(s) = -4as^{p-3}((p-2)\log s + 1)\),
namely \(s_2^* = e^{-1/(p-2)}\). Since \(\lim_{s\to 0+} V(s) = 0\)
and \(\lim_{s\to_\infty} V(s) = -\infty\), \(s_2^*\) is a global
maximum. For every \(s>0\),
\begin{linenomath}
    \begin{equation}
    \label{eq.Euler.2}
    \begin{split}
        G'(s) &= as^{p-1} (p\log|s|^2 + 2)\\[0.4em]
        G''(s) &= as^{p-2} (p(p-1)\log s^2 + 4p - 2)\\
    \end{split}
    \end{equation}
\end{linenomath}
Substituting \(G,G'\) and \(G''\) into the differential inequality, 
we obtain
\begin{linenomath}
    \begin{equation}
    \begin{split}
a^{-1} L(s) &= 12s^p\log s^2 - 7s^{p} (p\log s ^2 + 2)
+ s^{p} (p(p-1)\log s^2 + 4p - 2) \\
&=s^p\log s^2 (12 - 7p + p(p-1)) + 4s^p (p-4)\\
&= s^p (2(p-2)(p-6) \log s + 4(p-4)).
\end{split}
    \end{equation}
\end{linenomath}
We set \(A = 2(p-2)(p-6)\) and \(B = 4(p-4)\). Since $2<p<6$, we have $A<0$. Therefore, \(L > 0\) in a neighbourhood of zero which does not contain zero, and vanishes exactly once on \((0,+\infty)\) at the point
\(s^*\) such that
\begin{linenomath}
    \begin{equation}
        \log s^* = -\frac{2(p-4)}{(p-2)(p-6)}.
    \end{equation}
\end{linenomath}
We conclude the proof by showing that \(\log s^* > \log s_2^*\) or,
equivalently
\begin{linenomath}
    \begin{equation}
    \begin{split}
-\frac{2(p-4)}{(p-2)(p-6)} + \frac{1}{p-2} &= 
\frac{1}{p-2}\left(-\frac{2(p-4)}{(p-6)} + 1\right)\\
 &= \frac{1}{p-2}\cdot \frac{-2p + 8 + p-6}{p-6}\\
&=\frac{1}{p-2}\cdot \frac{2 - p}{p-6} = \frac1{6-p}>0.
\end{split}
    \end{equation}
\end{linenomath}
Finally, if \(p=6\), then \(L(s) = 8as^p\geq 0\).
\end{proof}
\begin{proof}[Proof of Theorem~\ref{thm.uniqueness-focusing}]
Let \(R_1\) and \(R_2\) be two minima of \(E\) on \(S(\lambda)\)
such that both \(R_1,R_2\) are positive and symmetric-decreasing.
There are Lagrange multipliers \(\mu_1,\mu_2\) such that 
\(E'(R_i) = \mu_i M'(R_i)\) for \(i=1,2\), or equivalently,
\(R_i\) is a weak solution to \(R_i'' - G'(R_i) + 2\mu_i R_i = 0\)
for \(i=1,2\). We set \(\omega_i := -2\mu_i\). From the
uniqueness result \cite[Theorem~5]{BL83a}, we have \(R_i = R_{\omega_i}\).
Without loss of generality, we can suppose that \(\omega_1\leq \omega_2\).
From Remark~\ref{rem.domain-log}, the domain of \(\lambda\)
is an interval. Since \(G\) satisfies the differential inequality \eqref{eq.Euler},
we apply \cite[Lemma~3.1]{GG17} to claim that \(\lambda\colon (0,\omega_2^*)\to (0,+\infty)\) is a monotonically non-decreasing function. Therefore, from \(\omega_1 < \omega_2\) it would
follow \(\lambda(\omega)\equiv\lambda\) for every \(\omega\in (\omega_1,\omega_2)\). Therefore, from \cite[Lemma~3.1]{GG17},
it follows \(L(s) = 0\) for every \(s\in (0,R_1(0))\). Therefore,
there are \(c_1,c_2\in\mathbb{R}\) such that
\(G(s) = c_1 s^2 + c_2 s^6\) for every \(s\in (0,R_1(0))\),
contradicting the definition of \(G\) given in \eqref{eq.log}.
If \(\omega_1 = \omega_2\), then we apply once again 
\cite[Theorem~5]{BL83a} to obtain \(R_1 = R_2\).
\end{proof}
\section{Stability of the non-linear Schr\"odinger equation with analytic non-linearity}
If \(a<0\) and \(2 < p < 6\), \cite[Lemma~3.1]{GG17} cannot
be applied as the sign change of \(a\) implies a sign change
in \eqref{eq.Euler}, meaning that we cannot rely on the monotonicity
property of \(\lambda\). Using analytic properties of
the mass function, however, it is possible
to show that \(\mathcal{G}_\lambda^{r,+}\) is a finite set,
even if not necessarily a singleton. The finiteness, however,
turns out to be enough to prove the orbital stability of 
standing-waves. 

In the next subsection we only assume that 
\(G\) is analytic on \((0,+\infty)\) 
and that the limits in \eqref{eq.ber-lio-str.1} are satisfied. Under
these assumptions alon domain of the mass function
is not necessarily an interval and its complement on
\((0,\sup(V))\) can be characterised as discontinuity points
of the function \(T\) defined in Proposition~\ref{prop.ber-lio-str}.
\subsection{Continuity of \(T\) and critical frequencies}
\label{ssect.frequencies}
\begin{proposition}
\label{prop.T-continuity}
Suppose that \(G\) is continuous and the first limit
in \eqref{eq.ber-lio-str.1} is satisfied.
Then \(T\colon (0,\sup(V))\to (0,+\infty)\) is a non-decreasing
piece-wise, continuous function with discontinuities occurring at \(\omega\) such that
\begin{enumerate}[(i).]
\itemsep=0.3em
    \item \(T(\omega)\) is a local maximum of \(V\) and
    \item there exists \(s>T(\omega)\) such that \(V(s)>\omega\).
\end{enumerate}
\end{proposition}
\begin{proof}
Firstly, we prove that \(T\) is non-decreasing. 
Given \(\omega_1\leq\omega_2\), we show that 
\(T(\omega_1)\leq T(\omega_2)\). Suppose that
\(T(\omega_1)>T(\omega_2)\) and set \(s_i := T(\omega_i)\) for \(i=1,2\).
Since \(V\) is continuous, \(V(0) = 0\) and \(V(s_1)\leq V(s_2)\), 
there exists \(s_3\in (0,s_2)\) such that \(V(s_3) = V(s_1)\). Since
\(s_3 < s_1\), \(V(s_3) = T(\omega_1)\) contradicts the minimality
of \(T(\omega_1)\) in the definition given in (ii) of 
Proposition~\ref{prop.ber-lio-str}.

(i) and (ii) imply that \(T\) is not continuous at \(\omega\).
Suppose that \(T(\omega)\) is a local maximum. Then, there exists \(\delta >0\) such
that \(V(s)\leq V(T(\omega))\) for every \(s\in (T(\omega) - \delta,T(\omega) + \delta)\).
From (ii), there exists \(s' >T(\omega)\) such that \(V(s')>V(T(\omega)) = \omega\). Then 
\(s'\geq T(\omega) + \delta\). We claim that 
for every \(\omega'\in (\omega,V(s'))\), there holds
\(T(\omega')\geq T(\omega) + \delta\). On the contrary,
we have \(T(\omega)\leq T(\omega') < T(\omega) + \delta\),
because \(T\) is non-decreasing. Since \(T(\omega)\)
is a global maximum in \((T(\omega) - \delta,T(\omega) + \delta)\), we have 
\(V(T(\omega'))\leq\omega\). That is \(\omega'\leq\omega\).

\paragraph{\(T\) is not continuous at \(\omega\) imply (i) and (ii)}
Let \((\omega_n)_{n\geq 1}\) be a sequence such that \(\omega_n > \omega\) and \(\delta>0\) such that \(\omega_n\to\omega\)
and \(T(\omega_n)\geq T(\omega) + \delta\). 

(i). If \(T(\omega)\) is not a local maximum,
there exists \(s\in (T(\omega) - \delta,T(\omega) + \delta)\) such that \(\omega < V(s)\);
there exists \(n_0\) such that \(V(s)>\omega_{n}>\omega\) for every \(n\geq n_0\).
By the Mean Value Theorem, there exists \(s_{n_0}\in (T(\omega),s)\) such that \(V(s_{n_0}) = \omega_{n_0}\).
By definition of \(T\), we have 
\begin{linenomath}
    \begin{equation}
    \delta + T(\omega)\leq T(\omega_{n_0})\leq s_{n_0} < s < \delta + T(\omega).
    \end{equation}
\end{linenomath}
If \(s<T(\omega)\), then a contradiction can be obtained from
\begin{linenomath}
    \begin{equation}
    \delta + T(\omega)\leq T(\omega_{n_0})\leq s_{n_0} < T(\omega).
    \end{equation}
\end{linenomath}
(ii) follows directly from the assumption that some \(\omega_n>\omega\). The inequality
\(T(\omega_n) < T(\omega)\) contradicts the monotonicity of \(T\), while \(T(\omega_n) = T(\omega)\)
implies \(\omega_n = \omega\). Therefore, we set \(s := T(\omega_n)\) and obtain (ii).
We refer to Figure~\ref{fig.T-not-continuous} for an example.
\end{proof}
\begin{figure}[h!]
\label{fig.T-continuous}
\centering
\includegraphics[width=0.7\textwidth]{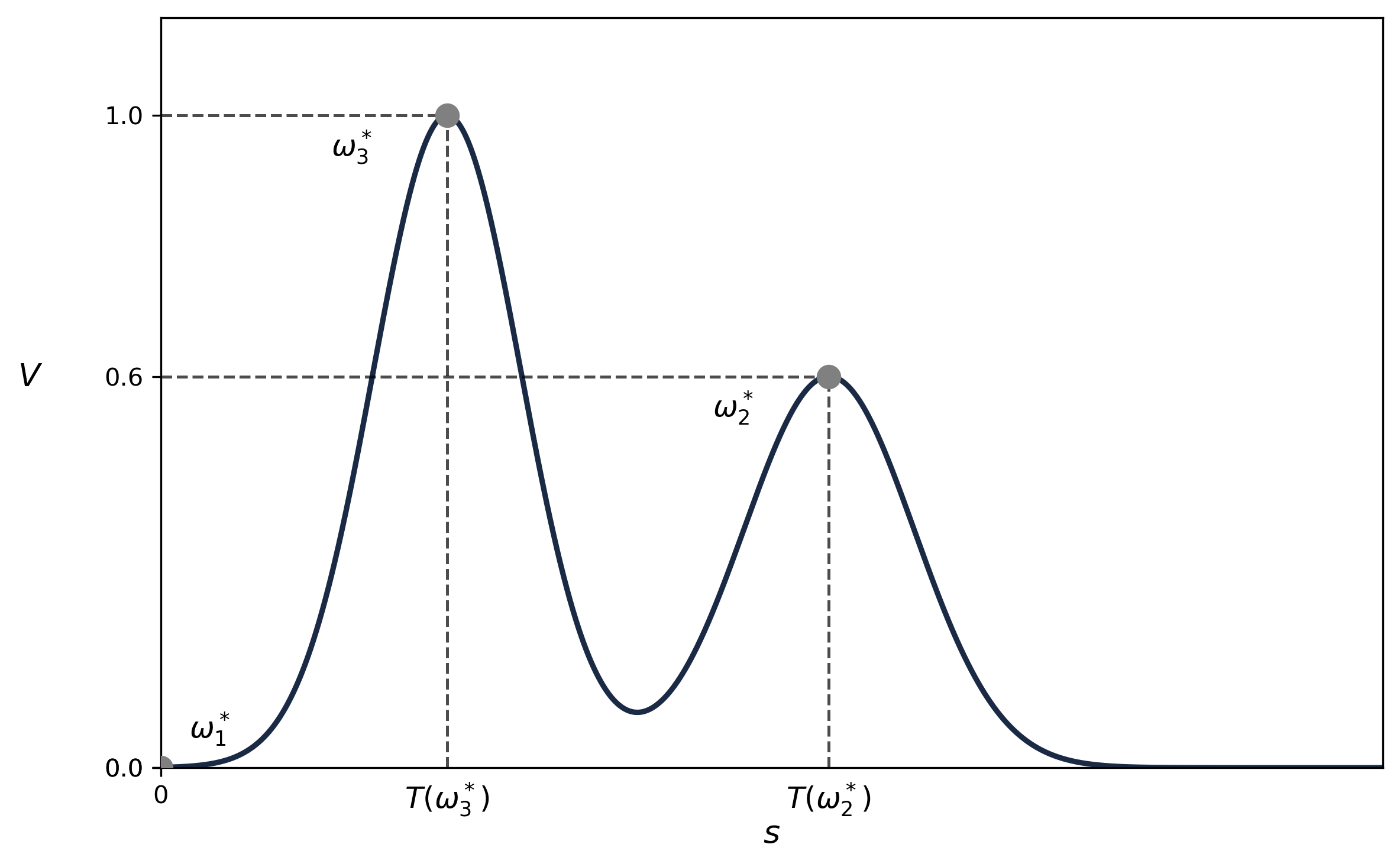}
\caption{\(V\) has two local maxima and \(T\) is continuous on
\((0,\sup(V))\).}
\label{fig:V_two_maxima_1}
\end{figure}
\begin{figure}[h!]
\centering
\includegraphics[width=0.7\textwidth]{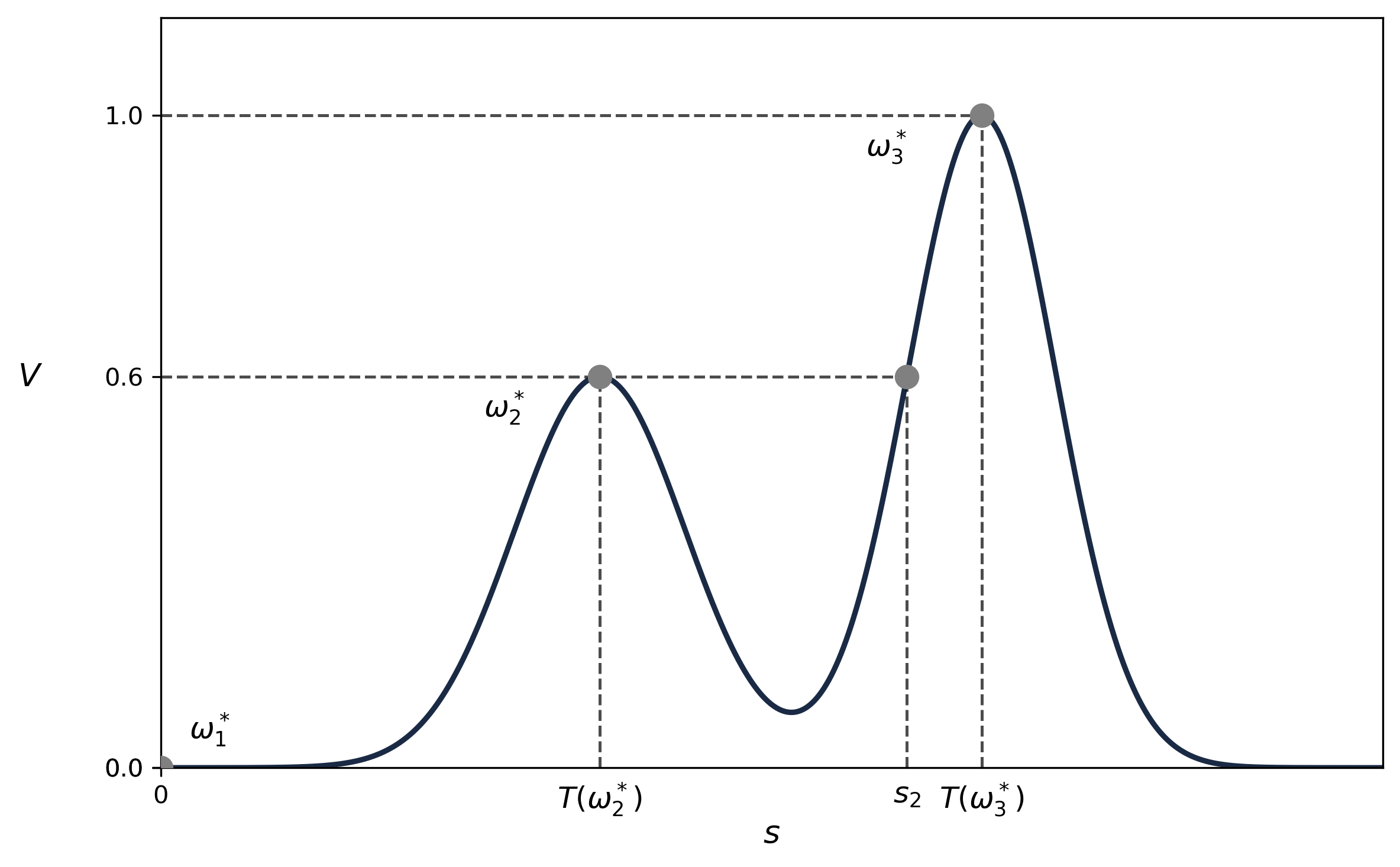}
\caption{\(V\) has two local maxima and \(T\) has a discontinuity
at \(\omega = \omega_2^*\).}
\label{fig.T-not-continuous}
\end{figure}
\subsection{Analytic properties of the mass function}
Throughout this section, we will assume that \(G\) and consequently \(V\).
are analytic on \((0,+\infty)\). Therefore, \(K_V\), the set of critical values of \(V\)
is isolated on \((0,+\infty)\) with \(0\) the only possible accumulation point. From 
Proposition~\ref{prop.T-continuity}, the set of discontinuity points of \(T\)
is a subset of \(K_V\) and it is a countable subset of \([0,\sup(V))\).
We use the representation
\begin{linenomath}
    \begin{equation}
    (\omega_k^*)_{k\in\mathbb{Z}},\quad \omega_k^* < \omega_{k+1}^*.
    \end{equation}
\end{linenomath}
In most of the applications, pure-power combinations or log-modified pure-powers,
\(0\) is not an accumulation point of \(K_V\). In this case, we prefer the representation
\begin{linenomath}
    \begin{equation}
      (\omega_k^*)_{k\geq 1},\quad \omega_k^* < \omega_{k+1}^*
    \end{equation}
\end{linenomath}
and denote \(\omega_1^* = 0\). We also denote
\begin{linenomath}
    \begin{equation}
        t_k := T(\omega_k^*),\quad s_k := \sup\big\{s|V(s) = \omega_k^*\big\}\cap (0,T(\omega_{k+1}^*)).
    \end{equation}
\end{linenomath}
Clearly \(t_k<s_k<t_{k+1}\) (check Figure~\ref{fig.T-not-continuous}).
\begin{proposition}
\label{prop.local-integrability}
For every \(k\geq 1\) the function \(T\colon (\omega_k^*,\omega_{k+1}^*)\to (s_k,t_{k+1})\) is 
the inverse of \(V\) and \((T^2)'\) is in \(L^1(\omega_k^*,\omega_{k+1}^*)\).
\end{proposition}
\begin{proof}
From (ii) of Proposition~\ref{prop.ber-lio-str}, 
\(V(T(\omega)) = \omega\). By definition of \(s_k,t_{k+1}\), \(V(s_k) = \omega_k^*\)
and \(V(t_{k+1}) = \omega_{k+1}^*\). The function \(V\) is strictly increasing
on the interval \((s_k,t_{k+1})\)
On the contrary, let 
\(s_k<s_M<t_{k+1}\) be a local maximum of \(V\). If \(V(s_M)\geq\omega_{k+1}^*\), since \(V(0)=0\),
from the Mean Value Theorem we obtain a contradiction with the minimality of \(T(\omega_{k+1}^*)\).
Therefore, \(V(s_M)<\omega_{k+1}^*\). If \(V(s_M)\leq\omega_{k}^*\),
by the Mean Value Theorem on the interval \([s_M,t_{k+1}]\), there exists
\(s>s_M\) such that \(V(s) = \omega_k^*\), contradicting the maximality of \(s_k\).
Therefore, \(V(s_M)\) satisfies (i) and 
(ii) in Proposition~\ref{prop.T-continuity} and is a discontinuity point of \(T\),
contradicting the definition of \(\omega_k^*\). Therefore, 
\(V\colon (s_k,t_{k+1})\to (\omega_k^*,\omega_{k+1}^*)\) is injective and surjective and the composition \(T\circ V = id\) also holds. To prove that \((T^2)'\) is \(L^1\), we consider \(\eta>0\) such that
\(\omega_k^* + \eta < \omega_{k+1}^* - \eta\) and evaluate
\begin{linenomath}
\begin{equation}
\begin{split}
\int_{\omega_{k}^* + \eta}^{\omega_{k+1}^* - \eta} (T^2)'(y)dy = 
T(\omega_{k}^* + \eta)^2 - T(\omega_{k+1}^* - \eta)^2.
\end{split}
\end{equation}
\end{linenomath}
Since \(T\) is an increasing function, the term above is the sum of two
increasing functions of \(\eta\). Since \(T\) is bounded on \((\omega_k^*,\omega_{k+1}^*)\), the limit
exists as \(\eta\to 0\). Since \((T^2)'\geq 0\) on \((\omega_k^*,\omega_{k+1}^*)\), 
we proved that \((T^2)'\in L^1(\omega_k^*,\omega_{k+1}^*)\).
\end{proof}

\begin{lemma}
\label{lemma.lambda_analytic}
Suppose that \(G\in C^\omega ((0,+\infty),\mathbb{R})\), and
\eqref{eq.ber-lio-str.1} are satisfied. Then the mass function
\begin{linenomath}
    \begin{equation}
        \lambda\colon\Omega\to (0,+\infty),\quad
        \lambda(\omega) := \|R_\omega\|_{L^2}^2
    \end{equation}
\end{linenomath}
is analytic. Moreover, if
\begin{linenomath}
    \begin{equation}
    \label{T}
\omega^{\frac12}(T^2)'(\omega) = o(1)
    \end{equation}
\end{linenomath}
as \(\omega\to 0\), then \(\lambda=o(1)\) as \(\omega\to 0\).
\end{lemma}
\begin{proof}
For every \(\omega\in\Omega\), we have
\begin{linenomath}
\begin{equation*}
\begin{split}
\lambda(\omega) = \int_{-\infty}^{+\infty} R_\omega(x)^2 dx = 
2\int_{0}^{+\infty} R_\omega(x)^2 dx = 2\int_{0}^{+\infty} \frac{R_\omega(x)^2 R_\omega'(x)}{R_\omega'(x)} dx.
\end{split}
\end{equation*}
\end{linenomath}
From (i) of Proposition~\ref{prop.ber-lio-str}, \(R'\) vanishes only at \(x = 0\). Therefore, the integral above is well defined. Multiplying \eqref{eq.log} in dimension one by \(2R'\)
yields 
\begin{linenomath}
\begin{equation}
\frac{d}{dx} \left(R'(x)^2 - 2G(R(x)) - \omega R(x)^2\right)\equiv 0.
\end{equation}
\end{linenomath}
From \cite[Remark~6.3]{BL83a}, \(R',2G,R\) have exponential decay. Therefore,
\begin{linenomath}
    \begin{equation}
R'(x)^2 - 2G(R(x)) - \omega R(x)^2\equiv 0.
    \end{equation}
\end{linenomath}
Since \(R'<0\), we have        
\begin{equation}
\label{eq.lambda.2}   
R_\omega'(x)=-\sqrt{\omega R_\omega(x)^2+2G(R_\omega(x))} = -R_\omega (x)\sqrt{\omega - V(R_\omega(x))}.
\end{equation}
Using \eqref{eq.lambda.2} the last
term of the equality above is equal to
\begin{linenomath}
\begin{equation*}
-2\int_{0}^{+\infty} \frac{R_\omega(x)^2 R_\omega'(x)dx}{R_\omega (x)\sqrt{\omega - V(R_\omega(x))}}
= -2\int_{0}^{+\infty} \frac{R_\omega(x) R_\omega'(x)dx}{\sqrt{\omega - V(R_\omega(x))}}.
\end{equation*}
\end{linenomath}
From (ii) of Proposition~\ref{prop.ber-lio-str}, \(R_\omega(0) = T(\omega)\) and
\(R_\omega\colon [0,+\infty)\to (0,T(\omega)]\) is strictly decreasing and surjective. 
Given \(\omega_0\in\Omega\), we want to show that \(\lambda\) has a power series
expansion in a neighbourhood of \(\omega_0\). Let \(n\) be the largest integer such that \(\omega_{n}^* < \omega_0\). The variable change \(R_\omega(x) = s\) yields
\begin{linenomath}
\begin{equation}
\label{eq.lambda.3}
\begin{split}
2\int_0^{T(\omega)}\frac{sds}{\sqrt{\omega - V(s)}} &= 
2\int_0^{s_n}\frac{sds}{\sqrt{\omega - V(s)}}
+ 2\int_{s_n}^{T(\omega)}\frac{sds}{\sqrt{\omega - V(s)}} \\
&= 
2\int_0^{s_n}\frac{sds}{\sqrt{\omega - V(s)}}
+ \int_{\omega_n^*}^{\omega} \frac{(T^2)'(y)dy}{\sqrt{\omega - y}}\\
&=: A(\omega) + B(\omega). 
\end{split}
\end{equation}
\end{linenomath}
Proposition~\ref{prop.local-integrability} allows
the variable change \(s = T(y)\) on the interval \((s_n,T(\omega))\). 
From Proposition~\ref{prop.local-integrability} it also follows that \((T^2)'\)
is analytic on \(\Omega\), because it is the inverse function of \(V\).
Then, there exists \(r > 0\) and a sequence \((c_n)_{n\geq 1}\) such that 
\begin{linenomath}
\begin{equation}
\label{eq.lambda.4}
(T^2)'(y) = \sum_{n=0}^{\infty} c_n(y - \omega_0)^n,\quad y\in (\omega_0 - r,\omega_0 + r).        
\end{equation}
\end{linenomath}
In fact, up to refine the choice of \(r\), we can suppose
that \eqref{eq.lambda.4} converges absolutely at \(y=\omega_0 + r\)
and \(\omega_0 + r<\omega_n^*\).
\subsubsection*{\(A\) is analytic at \(\omega_0\)}
Since \(V(s)\neq \omega_0\) for every \(s\in [0,s_n]\), we can write
\begin{linenomath}
\begin{equation}
\begin{split}
\label{eq.radical-expansion}
\sqrt{\omega - V(s)} &= \sqrt{\omega - \omega_0 + \omega_0 - V(s)} \\
&= \sqrt{\omega_0 - V(s)}\cdot\sqrt{1 + \frac{\omega - \omega_0}{\omega_0 - V(s)}}
\end{split}
\end{equation}
\end{linenomath}
and
\begin{equation}
\label{eq.lambda.5}
\omega_0 - V(s)\geq \omega_0 - \omega_n^*.
\end{equation}
Therefore,
\begin{linenomath}
    \begin{equation}
        \label{eq.lambda.8}
\begin{split}
A(\omega) &= 2\int_{0}^{s_n}
\frac{s}{\sqrt{\omega_0 - V(s)}} \sum_{n=0}^{\infty}{-\frac{1}{2}\choose n}
\frac{(\omega - \omega_0)^n}{(\omega_0 - V(s))^n} ds \\
&= \sum_{n=0}^{\infty}\left(2\int_{0}^{s_n}
\frac{sds}{(\omega_0 - V(s))^{n + \frac12}}
\right){-\frac{1}{2}\choose n}(\omega - \omega_0)^n \\
\end{split}
\end{equation}
\end{linenomath}
From Proposition~\ref{prop.local-integrability}, the
H\"older inequality and \eqref{eq.lambda.5},
\begin{linenomath}
\begin{equation}
\begin{split}
2\int_{0}^{s_n}
\frac{sds}{(\omega_0 - V(s))^{n + \frac12}}&\leq
2\int_{0}^{s_n}\frac{sds}{(\omega_0 - \omega_n^*)^{n + \frac12}}
= \frac{s_n^2}{(\omega_0 - \omega_n^*)^{n + \frac12}}
\end{split}
\end{equation}
\end{linenomath}
Therefore,
\(
A(\omega) = \sum_{n=0}^\infty a_n (\omega - \omega_0)^n,
\)
where 
\begin{linenomath}
    \begin{equation}
    \begin{split}
        |a_n|&\leq 
        \frac{s_n^2}{(\omega_0 - \omega_n^*)^{n + \frac12}}
        \left|{-\frac{1}{2}\choose n}\right| = 
        \frac{s_n^2}{(\omega_0 - \omega_n^*)^{n + \frac12}}\cdot
        \frac{(2n)!}{4^n n! n!}\\
        &\leq 
        \frac{s_n^2}{(\omega_0 - \omega_n^*)^{n + \frac12}}\cdot
        \frac{1}{\sqrt{\pi n}}. 
        \end{split}
    \end{equation}
\end{linenomath}
If we choose \(\delta < \omega_0 - \omega_{n}^*\), the series in \eqref{eq.lambda.8} converges if \(|\omega - \omega_0|<\delta\).
\subsubsection*{\(B\) is analytic at \(\omega_0\)} 
On \((\omega_n^*,\omega)\)
we cannot rely on the analytic properties of \(\sqrt{\omega - y}\).
Instead, we rely on \eqref{eq.lambda.4}. 
However, the convergence radius of \((T^2)'\)
may not cover the entire interval \((\omega_n^*,\omega_{n+1}^*)\). 
Therefore, it is convenient to write
\begin{linenomath}
\begin{equation*}
\begin{split}
B(\omega) &= \int_{\omega_n^*}^{\omega}
\frac{(T^2)'(y)dy}{\sqrt{\omega - y}} = 
\int_{\omega_n^*}^{\omega - \frac{r}2}
\frac{(T^2)'(y)dy}{\sqrt{\omega - y}} + 
\int_{\omega - \frac{r}2}^{\omega}
\frac{(T^2)'(y)dy}{\sqrt{\omega - y}}\\
&=B_1(\omega) + B_2(\omega).
\end{split}
\end{equation*}
\end{linenomath}
We also require
\begin{linenomath}
    \begin{equation}
        \label{eq.choice.2}
        \delta < \frac{r}4.
    \end{equation}
\end{linenomath}
The term \(B_1(\omega)\) can be dealt with in a similar fashion
to \(A\). From \eqref{eq.radical-expansion},
\begin{linenomath}
    \begin{equation}
        \label{eq.lambda.7}
\begin{split}
B_1(\omega) &= \int^{\omega - \frac{r}2}_{\omega_n^*}
(\omega_0 - y)^{-\frac12}(T^2)'(y)\sum_{n=0}^{\infty}{-\frac{1}{2}\choose n}
\frac{(\omega - \omega_0)^n}{(\omega_0 - y)^n} dy \\
&= \int^{\omega - \frac{r}2}_{\omega_n^*}
(\omega_0 - y)^{-\frac12}(T^2)'(y)\sum_{n=0}^{\infty}
{-\frac{1}{2}\choose n}\cdot
\frac{(\omega - \omega_0)^n}{(\omega_0 - y)^n} dy\\
&= \sum_{n=0}^{\infty} 
{-\frac{1}{2}\choose n}
\int^{\omega - \frac{r}2}_{\omega_n^*} (T^2)'(y) (\omega_0 - y)^{-\frac12 - n} 
(\omega - \omega_0)^n dy.
\end{split}
\end{equation}
\end{linenomath}
From Proposition~\ref{prop.local-integrability}, \eqref{eq.choice.2}, the
H\"older inequality and \eqref{eq.lambda.5}, 
\begin{linenomath}
\begin{equation}
\begin{split}
\int^{\omega - \frac{r}2}_{\omega_n^*} (T^2)'(y) 
(\omega_0 - y)^{-\frac12 - n}dy
&\leq \|(T^2)'\|_{L^1(\omega_n^*,\omega_{n+1}^*)} 
\|(\omega_0 - y)^{-\frac12 - n}\|_{L^{\infty}(
\omega_n^*,\omega - \frac{r}2)}\\
&= \|(T^2)'\|_{L^1(\omega_n^*,\omega_{n+1}^*)} \left(\frac{r}{4}\right)^{-\frac12 - n}.
\end{split}
\end{equation}
\end{linenomath}
The last inequality follows from \(y\leq\omega_0 - r/4\).
Therefore, \(B_1\) admits the expansion
\(
B_1(\omega) = \sum_{n=0}^\infty b_n^{(1)} (\omega - \omega_0)^n,
\)
where 
\begin{linenomath}
    \begin{equation}
    \begin{split}
        |b_n^{(1)}|&\leq 
        \|(T^2)'\|_{L^1(\omega_n^*,\omega_{n+1}^*)} \left(\frac{r}{4}\right)^{-\frac12 - n}
        \left|{-\frac{1}{2}\choose n}\right|\\
        &\leq 
        \|(T^2)'\|_{L^1(\omega_n^*,\omega_{n+1}^*)} \left(\frac{r}{4}\right)^{-\frac12 - n}\frac{1}{\sqrt{\pi n}}.
        \end{split}
    \end{equation}
\end{linenomath}
The convergence radius is \(r/4\). We conclude this part with the proof that
\(B_2\) is also analytic. Since \(y\in(\omega - r/2,\omega)\), we have
\begin{linenomath}
    \begin{equation}
        \omega_0 - \frac{r}{4} - \frac{r}{2}<\omega - \frac{r}2\leq y\leq \omega < \omega_0 + \frac{r}4
    \end{equation}
\end{linenomath}
implying \(|y - \omega_0| < 3r/4 < r\). From \eqref{eq.lambda.4},
\begin{linenomath}
\begin{equation*}
\begin{split}
B_2 (\omega) &= \int_{\omega - \frac{r}{2}}^{\omega}
\frac{(T^2)'(y)dy}{\sqrt{\omega - y}}
= \int_{\omega - \frac{r}{2}}^{\omega} (\omega - y)^{-\frac12}\sum_{n=0}^{\infty} c_n(y - \omega_0)^n dy \\
&= \sum_{n=0}^{\infty} c_n 
\int_{\omega - \frac{r}{2}}^{\omega}
(\omega - y)^{-\frac12}(y - \omega_0)^n dy \\
&= \sum_{n=0}^{\infty} c_n\int_0^{\frac{r}2} x^{-\frac12}((\omega - \omega_0) - x)^n dx =: \sum_{n=0}^\infty p_n(\omega)
\end{split}
\end{equation*}
\end{linenomath}
Each term of the series is a polynomial. Moreover,
\begin{linenomath}
\begin{equation}
\begin{split}
|p_n(\omega)|&\leq |c_n|\|x^{-\frac12}\|_{L^1(0,r/2)}\cdot\||(\omega - \omega_0) - x|^{n}\|_{L^\infty(0,r/2)}\\
&\leq 2^n |c_n|\left(\frac{r}{2}\right)^{\frac12} 
\||\omega - \omega_0|^n + x^n\|_{L^\infty(0,r/2)}\\
&\leq 2^{n+1} |c_n|\left(\frac{r}{2}\right)^{n + \frac12} = 
2^{-\frac12}|c_n| r^n.
\end{split}
\end{equation}
\end{linenomath}
where both \(x\) and \(|\omega - \omega_0|\) have been estimated with
\(r/2\). The series of the last term converges, given our assumptions
on \(r\). By the Weiestrass criterion, the series of
\((p_n)_{n\geq 1}\) converges uniformly on
\(\overline{B}(\omega_0,r/2)\), implying that \(B_2\) is analytic at \(\omega_0\).

We prove the convergence of \(\lambda\) to zero. Given \(\varepsilon>0\),
we choose \(0<\omega < \omega_2^*\). Since \(T\) is invertible
on \((0,\omega_2^*)\), by the H\"older inequality, we have
\begin{linenomath}
    \begin{equation}
        \begin{split}
            \lambda(\omega) := \int_0^\omega \frac{(T^2)'(y)dy}{\sqrt{\omega - y}}
            &\leq \|(T^2)'\|_{L^{\infty}(0,\omega)}\|(\omega - y)^{-\frac12}\|_{L^{1}(0,\omega)} \\
            &= 2\|(T^2)'\|_{L^{\infty}(0,\omega)}\cdot\omega^{\frac12}
            \sim\omega^{\frac12}(T^2)'(\omega) = o(1).
        \end{split}
    \end{equation}
\end{linenomath}
\end{proof}
\begin{remark}
We cannot expect \(\lambda\) to be analytic at \(\omega = 0\).
In fact, for a class of algebraic non-linearities studied in \cite[Proposition~2.3]{GG25}, one has \(\lambda\sim\omega^{\frac12}\) in a neighbourhood of \(\omega=0\). 
Assumption \eqref{T} is optimal in the sense that if one replaces it with 
\(\omega^r(T^2)'(\omega)=o(1)\) and \(r>\frac 12\), then \(\lambda\) might not converge to zero. 
One can consider \(G(s) = -a|s|^6\)
for any \(a>0\), where \(\lambda\) is constant. We
refer to \cite[\S3.2]{LR20} for an exact form of \(\lambda\) in the pure power case \(G = -a|s|^p\) with \(2<p\leq 2 + \frac4n\):
\begin{linenomath}
    \begin{equation}
    \lambda(\omega) = k\omega^{\frac{4+n(2-p)}{2(p-2)}}
    \end{equation}
\end{linenomath}
for some \(k=k(a)\). Since \(V = 2a|s|^{p-2}\),
\begin{linenomath}
    \begin{equation}
        \begin{split}
        T(\omega) &= \left(\frac{\omega}{2a}\right)^{\frac1{p-2}}\\
        (T^2)'(\omega) &= \frac2{p-2}\left(\frac{1}{2a}\right)^{\frac2{p-2}}\omega^{\frac{4-p}{p-2}}.
        \end{split}
    \end{equation}
\end{linenomath}
In the case \(n=1\), we have
\begin{linenomath}
    \begin{equation}
        \frac{4 - p}{p - 2} = -\frac12\iff p=6 \iff\lambda(\omega) \equiv k(a).
    \end{equation}
\end{linenomath}
\end{remark}
\begin{remark}
One would expect to obtain the analyticity as a result of the
analyticity of the function 
\begin{linenomath}
\begin{equation}
\label{eq.phi}
\phi\colon\Omega\to H^1_r,\quad\omega\to R_\omega
\end{equation}
\end{linenomath}
and \(M\colon H^1_r\to [0,+\infty)\).
Indeed, the first function has been proved to be at least \(C^1(\Omega,H^1_r)\) in dimension \(n\geq 3\), provided
the operator 
\begin{linenomath}
\begin{equation}
\mathscr{L}(\phi) := \Delta \phi - G''(R_\omega)\phi - \omega \phi
\end{equation}
\end{linenomath}
is such that \(\ker(\mathscr{L})\cap H^1_r = \{0\}\) (we refer to \cite[Lemma~20]{SS85} for a proof), and in dimension \(n=1\) under very general assumptions \cite[Proposition~6]{Gar23}, as the non-degeneracy required in \cite{SS85} is satisfied. Unfortunately, it is not straightforward to obtain
higher regularity of \(\phi\). In fact, the arguments used in \cite[Lemma~20]{SS85} and \cite[Proposition~6]{Gar23} rely on the regularity 
of the function \(H^1_r\ni u\to (\omega^2 - \Delta)^{-1} G'(u)\in H^1_r\). This non-linear operator can be proved to be \(C(H^1_r,H^1_r)\) in the quoted lemma, 
but higher regularity of this operator does not just depend on
the regularity of \(G\): if \(k\geq 2\), the most natural candidate
for the differential is the map \(H^1_r\ni v\to G''(u)v\), which in general
is not in \(H^1_r\) if \(G\) is as in \eqref{eq.log} and \(p<4\). 
To give the idea of the rigidity to be imposed on \(G\) that makes 
Nemytskii operators analytic, we refer to \cite[Theorem~6.2]{ABM14}: given
\(p,q\geq 1\) and \(U\) bounded interval, the map 
\begin{linenomath}
    \begin{equation}
\mathcal{N}_{G'}\colon L^p(U)\ni u\to G'\circ u\in L^q(U)
    \end{equation}
\end{linenomath}
is well-defined if and only if 
\(G'\) is a polynomial of degree at most \(p/q\).
\end{remark}
\subsection{Stability of standing-waves for analytic non-linearities}
In this subsection we prove stability of standing-waves for a non-linearity \(G\)
satisfying general assumptions for the existence of normalized standing-waves on \(S(\lambda)\) for 
\(\lambda\) large enough.




\begin{lemma}
\label{lem.gap}
\(d(\mathcal{G}_\lambda(R_1),\mathcal{G}_\lambda(R_2))\geq\|R_1 - R_2\|_{2}\) 
for every \(R_1,R_2\) in \(\mathcal{G}_{\lambda}^{r,+}\).
\end{lemma}
\begin{proof}
Given \((z_1,y_1),(z_2,y_2)\in S^1\times\R\)
\begin{linenomath}
    \begin{equation}
    \begin{split}
        \|z_1 R_1 (\cdot + y_1) - z_2 R_2(\cdot + y_2)\|_{H^1} &= 
        \|R_1 (\cdot + y_1) - z_2 \overline{z}_1 R_2(\cdot + y_2)\|_{H^1} \\[0.4em]
        &=\|R_1 - z_2 \overline{z}_1 R_2(\cdot + y_2 - y_1)\|_{H^1}\\[0.4em]
        &\geq \|R_1 - z_2 \overline{z}_1 R_2(\cdot + y_2 - y_1)\|_{2}\\[0.4em]
        &=g(\alpha,\beta,y)
        \end{split}
    \end{equation}
\end{linenomath}
where \(\alpha + i\beta = z_2 \overline{z}_1\), \(y = y_2 - y_1\) and 
\(g\colon S^1\times\R\to\R\) is the function defined as 
\(g(\alpha,\beta,y) := \|R_1 - z R_2(\cdot + y)\|_{2}\).
Since \(g\geq 0\), the function is bounded below. Moreover, from the definition of inner product \eqref{eq.inner}, \(g(\alpha,\beta,y)^2\) can be rewritten as
\begin{linenomath}
    \begin{equation}
    \label{eq.gap.1}
    \begin{split}
g^2 &= \|R_1 - \alpha R_2(\cdot + y)\|_{2} ^2 + \beta^2 \|R_2 (\cdot +y)\|_{2}^2 \\
&= \|R_1\|_{2}^2 - 2\alpha (R_1,R_2(\cdot + y))_{2} + \alpha^2\|R_2 (\cdot + y)\|_{2}^2 + 
\beta^2 \|R_2 (\cdot + y)\|_{2}^2 \\
&= \|R_1\|_{2}^2 - 2\alpha (R_1,R_2(\cdot + y))_{2} + \|R_2 (\cdot + y)\|_{2}^2 \\
&= \|R_1\|_{2}^2 - 2\alpha (R_1,R_2(\cdot + y))_{2} + \|R_2\|_{2}^2.
\end{split}
    \end{equation}
\end{linenomath}
Therefore, 
\begin{linenomath}
    \begin{equation}
    \lim_{|y|\to\infty} g^2(\alpha,\beta,y) = \|R_1\|_{2}^2 + \|R_2\|_{2}^2 > \|R_1 - R_2\|_{2}^2 = 
    g^2(1,0,0).
    \end{equation}
\end{linenomath}
as \(R_1,R_2>0\) from (i) of Proposition~\ref{prop.ber-lio-str}.
Then, the minimum of \(g^2\) on \(S^1\times\R\) is achieved. If \((\alpha,\beta,y)\) is
a constrained critical point of \(g^2\) on \(S^1\times\R\), then there exists \(\mu\in\R\) such
that 
\begin{linenomath}
    \begin{equation}
\nabla g^2(\alpha,\beta,y) = 2\mu(\alpha,\beta,0)^T.
    \end{equation}
\end{linenomath}
That is
\begin{linenomath}
\begin{equation}
\begin{split}
-2(R_1,R_2(\cdot +y))_{2} &= 2\mu\alpha\\[0.4em]
0 &= 2\mu\beta\\[0.3em]
-2\alpha(R_1,R_2'(\cdot + y))_{2} &= 0.
\end{split}
\end{equation}
\end{linenomath}
From the first equation, \(\alpha,\mu\neq 0\). On the contrary, 
\begin{linenomath}
    \begin{equation}
    (R_1,R_2(\cdot + y))_2 = 0,
    \end{equation}
\end{linenomath}
which implies \(R_1\equiv 0\) and contradicts \(\lambda>0\). 
Since \(\mu\neq 0\), we have \(\beta = 0\). Therefore, \(\alpha = 1\) and 
\((R_1,R_2'(\cdot + y))_{2} = 0\). Let \(y_0\) be a solution to this equation (at least one
solution exists, namely \(y_0=0\)). From \eqref{eq.gap.1}, 
\begin{linenomath}
    \begin{equation}
    \begin{split}
g^2(\alpha,\beta,y)&\geq g^2(1,0,y_0) = \|R_1 - R_2(\cdot + y_0)\|_{2} ^2\\[0.4em]
&\geq \|R_1^* - R_2(\cdot + y_0)^*\|_{2} ^2 = \|R_1 - R_2\|_{2} ^2 \\[0.4em]
&= g^2(1,0,0).
\end{split}
    \end{equation}
\end{linenomath}
The second inequality follows from the non-expansivity property of the Schwarz symmetrization,
\cite[Theorem~3.4]{LL01} and (i) of Proposition~\ref{prop.ber-lio-str}.
\end{proof}
\begin{proof}[Proof of Theorem~\ref{thm.finite}]
From (\hyperlink{G2a}{G2a}) it follows \(|G(s)|\leq C(|s|^p + |s|^q)\). Since \(p>2\),
the first limit in \eqref{prop.ber-lio-str.2} is satisfied. Still from (\hyperlink{G2a}{G2a}),
\begin{linenomath}
    \begin{equation}
\frac{|G'(s)|}{s}\leq C(|s|^{p-2} + |s|^{q - 2}).
    \end{equation}
\end{linenomath}
Since \(p-2>0\), the second limit also follows.

Suppose that there exist \(\lambda_0>0\) such that there
are infinitely many minimizers \(R_n\in\mathcal{G}_{\lambda_0}^{r,+}\)
defined in \eqref{eq.radial}. From
(ii) of Proposition~\ref{prop.ber-lio-str}, there exists \((\omega_n)_{n\geq 1}\)
such that \(R_n = R_{\omega_n}\). 
From \cite[Theorem~1.1]{GG17}, a subsequence of \((R_{\omega_n})_{n\geq 1}\) converges in \(H^1\)
to \(R\in S(\lambda_0)\) and \(E(R) = I(\lambda_0)\).
Therefore, \(R\in\mathcal{G}_{\lambda_0}^{r,+}\) and there exists
\(\omega_0\) such that \(R = R_{\omega_0}\).
The convergence in \(H^1\) implies 
\(\lim_{n\to\infty} R_{\omega_n}(0)\to R(0)\). Since
\(V\) is continuous, \(\lim_{n\to\infty} \omega_n = \omega_0\).
Now, we show that \(\omega_0\neq\omega_k^*\) for every
\(1\leq k\). On the contrary, \(\omega_0 = \omega_k^*\).
If \(k=1\), then \(\omega_k^* = 0\),
implying \(\omega_0 = 0\). Therefore, \(R\equiv 0\), contradicting
\(\lambda_0>0\). If \(k\geq 2\), then \(V'(T(\omega_k^*))=\omega_k^*\). From (ii) of
Proposition~\ref{prop.ber-lio-str}, \(R(0) = T(\omega_k^*)\). From 
the definition of \(\Omega\) in 
Definition~\ref{defn.mass-function}, \(V'(R(0)) = 0\),
contradicting Proposition~\ref{prop.ber-lio-str}.

Let \(k\in\mathbb{N}\) be such that \(\omega_0\in (\omega_k^*,\omega_{k+1}^*)\).
From (\hyperlink{G0}{G0}), \(V\in C^\infty((0,+\infty))\) and from Lemma~\ref{lemma.lambda_analytic},
the mass function \(\lambda\colon(\omega_k^*,\omega_{k+1}^*)\to (0,+\infty)\) is analytic. 
Since \(\lambda(\omega_n) = \lambda_0\) for every \(n\geq 1\),
we have \(\lambda(\omega) = \lambda_0\). The elliptic equation
in \eqref{eq.log} can be restated as the following
equality in \(({H^1})'\)
\begin{linenomath}
\begin{equation}
    E'(R_{\omega}) = \frac{\omega}2 M'(R_{\omega}).
\end{equation}
\end{linenomath}
for every \(\omega\in (\omega_k^*,\omega_{k+1}^*)\). Since the function \(\phi\) defined
in \eqref{eq.phi} is \(C^1(\Omega,H^1)\), we have
\begin{linenomath}
\begin{equation}
    \Big\langle E'(R_{\omega}),\phi'(\omega)\Big\rangle = 
    \Big\langle\frac{\omega}2 M'(R_{\omega}),\phi'(\omega)\Big\rangle.
\end{equation}
\end{linenomath}
Therefore,
\begin{linenomath}
    \begin{equation}
    \frac{d}{d\omega} E\circ\phi = \frac{\omega}2 \frac{d}{d\omega} M\circ\phi = \frac{\omega}2 \lambda'(\omega) = 0.
    \end{equation}
\end{linenomath}
Therefore, \(E(R_\omega)\) is constant on \((\omega_k^*,\omega_{k+1}^*)\) and 
\begin{linenomath}
    \begin{equation}
        E(R_\omega) = E(R_{\omega_1}) = I(\lambda_0),\quad
        \lambda(\omega) = \lambda(\omega_1) = \lambda_0.
    \end{equation}
\end{linenomath}
Therefore, \(R_\omega\in\mathcal{G}_{\lambda_0}^r\) for
every \(\omega\) in \((\omega_k^*,\omega_{k+1}^*)\) and the sequence of \(R_{\omega_k^* + (\omega_{k+1}^* -\omega_k^*)/n}\) has frequencies converging to \(\omega_k^*\). But the existence of such a sequence has been ruled out at the beginning of the proof. 
\end{proof}
\begin{proof}[Proof of Theorem~\ref{thm.ground-state-log}]
By Lemma~\ref{lem.concentration-compactness}, the set \(\mathcal{G}_\lambda\)
is non-empty for every \(\lambda>0\). From \eqref{eq.Euler.2}, \(G''\) exists and
is continuous. Therefore, \(G'\) is locally Lipschitz. According to \cite[Theorem~3.5.1]{Caz03} and \cite[Example~3.2.4]{Caz03}, the equation \eqref{eq.NLS} is locally well-posed in \(H^1(\mathbb{R};\mathbb{C})\).
From \eqref{eq.cc.8}, if \((a,p)\in (0,+\infty)\times (2,6]\) or
\eqref{eq.coercive.8}, if \((a,p)\in (-\infty,0)\times (2,6)\), or
local solutions are bounded in \(H^1(\mathbb{R};\mathbb{C})\). Therefore,
local solutions can be extended to global solutions. We denote \(U_t(u)\) the solution at time \(t\) with initial datum \(u\in H^1\). Let \(\lambda\) be such that 
\(\mathcal{G}_\lambda\) is not stable. Then there exists 
a sequence \((u_n)_{n\geq 1}\), \(u_n\) in \(S(\lambda)\), a 
real sequence \(\seq{t}\)
and \(\varepsilon_0>0\) such that
\[
d(u_n,\mathcal{G}_\lambda)\to 0,\quad d(U_{t_n}(u_n),\mathcal{G}_\lambda)\ge \varepsilon_0.
\]
We set \(v_n := U_{t_n}(u_n)\). Since \(E\) and \(M\) are constant on \(U_t(u_n)\), we have
\[
E(U_{t_n}(u_n))=E(u_n),\quad M(U_{t_n}(u_n))=M(u_n).
\]
Since \(d(u_n,\mathcal{G}_\lambda)\to 0\) and \(E\) and \(M\) are continuous on \(H^1\), we obtain \(E(u_n)\to I(\lambda)\) and \(M(u_n)\to\lambda\). Hence 
\(\seq{v}\) is a minimizing sequence for \(E\) constrained to 
\(S(\lambda)\). By Lemma~\ref{lem.concentration-compactness}, up to extracting a subsequence, there exists a real sequence \(\seq{y}\) and \(v\in\mathcal{G}_\lambda\) such that \(v_n(\cdot+y_n)\to v\) in \(H^1(\mathbb{R};\mathbb{C})\) for every
\(\lambda>0\). Therefore,
\begin{linenomath}
    \begin{equation}
d(v_n,\mathcal{G}_\lambda)\le \|v_n-v(\cdot-y_n)\|_{H^1}\to 0,        
    \end{equation}
\end{linenomath}
contradicting \(d(v_n,\mathcal{G}_\lambda)\ge\varepsilon_0\). Therefore, \(\mathcal{G}_\lambda\) is stable.
\end{proof}
\begin{proof}[Proof of Theorem~\ref{thm.stability-wave}]
From Theorem~\ref{thm.finite}, the set \(\mathcal{G}_\lambda^{r,+}\) is finite. By \cite[Lemma~2.4]{GG17}, every \(u\in\mathcal{G}_\lambda\) can be written as \(u(x)=zR(x+y)\) for some 
complex number \(|z|=1\), \(y\in\mathbb{R}\) and \(R\in\mathcal{G}_\lambda^{r,+}\). Hence \(\mathcal{G}_\lambda\) is the union of finitely many distinct orbits
\begin{linenomath}
    \begin{equation}
    \label{eq.partition}
        \mathcal{G}_\lambda = \bigcup_{i=1}^{N} \mathcal{G}_\lambda(R_i),
    \end{equation}
\end{linenomath}
where \(R_i\in\mathcal{G}_\lambda^{r,+}\). Since \(\mathcal{G}_\lambda^{r,+}\) is finite,
from Lemma~\ref{lem.gap}, there exists \(\delta>0\) such that 
\begin{equation}
\label{dist.1}    
d(\mathcal{G}_\lambda(R_i),\mathcal{G}_\lambda(R_j))\geq \|R_i - R_j\|_{2} \geq 3\delta
\text{ for all } i \neq j.
\end{equation}
For each \(i\), we define 
\begin{linenomath}
    \begin{equation}
E_\delta^i := \left\{E(u)\mid u\in S(\lambda)\cap\partial B(\mathcal{G}_\lambda(R_i),\delta)\right\}.
    \end{equation}
\end{linenomath}
We claim that 
\begin{linenomath}
    \begin{equation}
    \label{eq.mountain-pass}
        E_\delta^i > I(\lambda)
    \end{equation}
\end{linenomath} 
for every \(i\).
Assume by contradiction that \(E_\delta^i = I(\lambda)\) for some \(i\). Then there exists a sequence \(\seq{u}\), \(u_n\in S(\lambda)\) such that 
\begin{linenomath}
    \begin{equation}
d(u_n,\mathcal{G}_\lambda(R_i)) = \delta\text{ and } E(u_n)\to I(\lambda).
\end{equation}
\end{linenomath}
Thus \((u_n)_{n\geq 1}\) is a minimizing sequence for \(E\) on \(S(\lambda)\). By 
\cite[Theorem~1.1]{GG17}, or Lemma~\ref{lem.concentration-compactness} (in the case of \eqref{eq.log-interaction}),
up to extract a subsequence, 
there exists \((y_n)_{n\geq 1}\) and \(u\in\mathcal{G}_\lambda\) such that \(u_n(\cdot+y_n)\)
converges to \(u\) in \(H^1\). By continuity of the distance, we have
\begin{equation}
\label{dist.2}    
d(u,\mathcal{G}_\lambda(R_i)) = \lim_{n\to\infty} d(u_n(\cdot+y_n),\mathcal{G}_\lambda(R_i)) = \delta.
\end{equation}
Since \(u\in\mathcal{G}_\lambda\), we have \(u\in\mathcal{G}_\lambda(R_j)\) for some \(j\). If \(j \neq i\), then by \eqref{dist.1}, \(d(u,\mathcal{G}_\lambda(R_i)) \ge 2\delta\), contradicting \eqref{dist.2}. Hence \(j = i\), i.e., \(u\in\mathcal{G}_\lambda(R_i)\), then \(d(u,\mathcal{G}_\lambda(R_i)) = 0\), contradicting \eqref{dist.2} since \(\delta>0\). Therefore \eqref{eq.mountain-pass} holds true for all \(i\).

Now we suppose that some orbit, \(\mathcal{G}_\lambda(R_0)\), is not stable. Then exists a sequence \((u_n)_{n\geq 1}\) in \(S(\lambda)\), \((t_n)_{n\geq 1}\) real sequence and \(\varepsilon_0>0\) such that
\begin{equation}
\label{dist.3}    
d(u_n,\mathcal{G}_\lambda(R_0))\to 0, \quad  d(U_{t_n}(u_n),\mathcal{G}_\lambda(R_0))\ge \varepsilon_0.
\end{equation}
We set \(v_n := U_{t_n}(u_n)\). Since \(d(u_n,\mathcal{G}_\lambda)\to 0\) from (ii),
\(d(v_n,\mathcal{G}_\lambda)\to 0\). Since \(\mathcal{G}_\lambda\) is stable, there exists \(R_1\) such that 
\begin{linenomath}
    \begin{equation}
d(v_n,\mathcal{G}_\lambda(R_1))\to 0.
\end{equation}
\end{linenomath}
From \eqref{dist.3}, \(R_1 \neq R_0\); from \eqref{dist.1}, 
\(d(\mathcal{G}_\lambda(R_1),\mathcal{G}_\lambda(R_0)) \ge 2\delta\). For \(n\geq n_0\), 
\(d(v_n,\mathcal{G}_\lambda(R_1)) \le \delta\). By the triangle inequality, we have 
\begin{equation}
\label{dist.4}
d(v_n,\mathcal{G}_\lambda(R_0)) \ge 2\delta. 
\end{equation}
Consider the continuous function 
\(\beta_n(t) := d(v_n,\mathcal{G}_\lambda(R_0))\). 
From \eqref{dist.3} and \eqref{dist.4}, we have \(\beta_n(0) \to 0\) and
\(\beta_n(t_n) \ge 2\delta\), respectively. Then for \(n\geq n_0\),
\begin{linenomath}
\begin{equation}
\beta_n(0) < \delta < 2\delta\leq \beta_n(t_n).
\end{equation}
\end{linenomath}
By the Intermediate Value Theorem, there exists \(t_n^*\in (0,t_n)\) such that \(\beta_n(t_*) = \delta\).
We set \(w_n := U_{t_*}(u_n)\). Then \(w_n \in S(\lambda)\) and \(d(w_n,\mathcal{G}_\lambda(R_0)) = \delta\). Moreover, from the conservation of the energy, \(E(w_n) = E(u_n)\). Taking the limit as \(n\to\infty\), 
we obtain
\begin{linenomath}
    \begin{equation}
E_\delta^0\leq E(w_n) = E(u_n) \to I(\lambda)
    \end{equation}
\end{linenomath}
and therefore a contradiction with \eqref{eq.mountain-pass}. Therefore, each \(\mathcal{G}_\lambda(R_i)\) is stable. To conclude the proof, we show that \(\mathcal{G}_\lambda(u)\) is stable for every
\(u\in\mathcal{G}_\lambda\). From \eqref{eq.partition}, there exists \(R_i\) such that \(u\in\mathcal{G}_\lambda(R_i)\). Therefore, \(\mathcal{G}_\lambda(u) = \mathcal{G}_\lambda(R_i)\).
\end{proof}
\begin{theo}
\label{thm.log-asymptotic}
Suppose that \(2<p<6\) and \(a\neq 0\), or \(p=6\) and \(a>0\). Then
\(\lim_{\omega\to 0}\lambda(\omega)=0\).
\end{theo}
\begin{proof}
We check that \eqref{T} is fulfilled by 
\(G(s)=a|s|^p\log|s|^2\). Since
\(V(s)= -4a s^{p-2}\log s^2\).
If \(a<0\), then \(s_1=1\). Since 
\(V'(1)=-4a>0\), from 
Proposition~\ref{prop.T-continuity},
\(T\) is continuous at \(\omega_1^*=0\)
and \(T(0)=1\). 

If \(a>0\), then \(T(0)=0\). Since \(0\) is not a local maximum, 
from Proposition~\ref{prop.T-continuity}, \(T\) is continuous at \(0\), that is
\begin{linenomath}
    \begin{equation}
    \label{eq.stability-analytic.1}
        \lim_{\omega\to 0} T(\omega) = 0.
    \end{equation}
\end{linenomath}
Since
\begin{linenomath}
    \begin{equation}
        V'(s)=-2as^{p-3}((p-2)\log s^2+2)\sim -2a(p-2)s^{p-3}\log s^2.
    \end{equation}
\end{linenomath}
as \(s\to 0^+\). We set \(\omega = V(s)\) and obtain
\(\omega = -2as^{p-2} \log s^2\). Therefore,
\begin{linenomath}
    \begin{equation}
        \omega^{\frac12}=\sqrt{2a}\, T(\omega)^{\frac{p-2}2}(-\log T(\omega)^2)^{\frac12}.
    \end{equation}
\end{linenomath}
From \(V'(T(\omega))T'(\omega) = 1\), it follows 
\begin{linenomath}
    \begin{equation}
    \label{eq.stability-analytic.2}
     V'(T(\omega)) (T^2)'(\omega) = 2T(\omega).
    \end{equation}
\end{linenomath}
Therefore,
\begin{linenomath}
    \begin{equation}
    \begin{split}
\omega^{\frac12}(T^2)'(\omega) &= \omega^{\frac12}\frac{2T(\omega)}{V'(T(\omega))}
\sim\frac{\sqrt{2a}\, T(\omega)^{\frac{p-2}2}(-\log T(\omega)^2)^{\frac12}\cdot 2T(\omega)}{-2a(p-2)T(\omega)^{p-3}\log T(\omega)^2}\\
&= \frac{2}{-\sqrt{2a}(p-2)}\cdot T(\omega)^{\frac{6-p}2} (-\log T(\omega)^2)^{-\frac12}.
\end{split}
\end{equation}
\end{linenomath}
From \eqref{eq.stability-analytic.1}, \eqref{T} is satisfied.
Therefore, \(\lambda(\omega) = o(1)\) as \(\omega\to 0\).
\end{proof}
Proposition~\ref{prop.combined-powers} addresses a case left open in \cite{GG19} where the stability of a combination of three powers was considered. In one particular case, 
\begin{linenomath}
    \begin{equation}
        G(s) = -a|s|^p + b|s|^q - c|s|^r,\quad a,b,c>0.
    \end{equation}
\end{linenomath}
The analiticity of \(\lambda\) allows us to give a positive answer, at least for the problem of stability for an arbitrary number of power combinations.
\begin{proposition}
\label{prop.combined-powers}
Let \((p_1,p_2,\dots,p_m)\in (2,+\infty)^m\) be such that \(p_i\leq p_{i+1}\)
for every \(1\leq i\leq m\) and \(\alpha\in\mathbb{R}^m\) such that
\(p_i < 6\) if \(\alpha_i<0\) for every \(1\leq i\leq m\). Assume, moreover that 
\(p_1<6\). Then
\begin{linenomath}
    \begin{equation}
G(s) := \sum_{i=1}^m \alpha_i |s|^{p_i}
\end{equation}
\end{linenomath}
satisfies (\hyperlink{G0}{G0}), (\hyperlink{G1}{G1}),
(\hyperlink{G2a}{G2a}) and (\hyperlink{G2b}{G2b}),
and \(\lim_{\omega\to 0}\lambda(\omega) = 0\).
\end{proposition}
\begin{proof}
(G0). \(G\) is clearly analytic on \((0,+\infty)\), \(C^2\) on \(\mathbb{R}\)
and therefore, \(G'\) is locally Lipschitz.

(G1). From \(\alpha_1(p_1 - 6)>0\), it follows \(\alpha_1<0\). Then
\(G\) is negative when \(s\to 0\).

(G2a). Since \(p_i>2\) for every \(1\leq i\leq m\), \(G(0)=0\).
If \(s>0\), there holds 
\begin{linenomath}
    \begin{equation}
        G'(s)=\sum_{i=1}^m \alpha_i p_i s^{p_i-1}.
    \end{equation}
\end{linenomath}
If \(0\leq s\leq 1\), there holds 
\begin{linenomath}
    \begin{equation}
|G'(s)|\leq \bigg(\sum_{i=1}^m |\alpha_i| p_i\bigg)|s|^{p_1-1}.
    \end{equation}
\end{linenomath}
If \(s\geq 1\),
\begin{linenomath}
    \begin{equation}
|G'(s)|\leq \bigg(\sum_{i=1}^m |\alpha_i| p_i\bigg)|s|^{p_m-1}.
    \end{equation}
\end{linenomath}
In conclusion, we have
\begin{linenomath}
    \begin{equation}
        |G'(s)|\leq \bigg(\sum_{i=1}^m |\alpha_i|p_i\bigg) 
        (|s|^{p_1-1} + |s|^{p_m-1}).
    \end{equation}
\end{linenomath}
(G2b). One can choose \(s_*=1\), \(p^*=\max\{j\mid p_j<6\}\) and 
\begin{linenomath}
    \begin{equation}
    G(s)\geq\sum_{p_j < 6}\alpha_j|s|^{p_j}\geq\bigg(\sum_{p_j<6} \alpha_j
    \bigg) |s|^{p^*}.
    \end{equation}
\end{linenomath}
Finally, we check that \(\lim_{\omega\to 0}\lambda(\omega)=0\). 
Since \(\alpha_1<0\), \(V\) is positive in a neighbourhood of the origin.
Therefore, \(T(0)=0\). From Proposition~\ref{prop.T-continuity},
\(T\) is continuous at \(\omega = 0\). Therefore, \(\lim_{\omega\to 0} T(\omega) = 0\). Then
\begin{linenomath}
    \begin{equation}
        \omega = V(T(\omega))\sim -2\alpha_1 T(\omega)^{p_1-2},
        \quad V'(T(\omega))\sim -2\alpha_1(p_1 - 2)
        T(\omega)^{p_1-3}
    \end{equation}
\end{linenomath}
as \(\omega\to 0\). From \eqref{eq.stability-analytic.2},
\begin{linenomath}
    \begin{equation}
    \begin{split}
\omega^{\frac12}(T^2)'(\omega) &= \omega^{\frac12}\frac{2T(\omega)}{V'(T(\omega))} = \omega^{\frac12}\frac{2T(\omega)}{-2\alpha_1(p_1 - 2)
        T(\omega)^{p_1-3}} \\
&= -\frac{1}{2\alpha_1(p_1-2)}\cdot \omega^{\frac12} T(\omega)^{4 - p_1}\\
&\sim
-\frac{1}{2\alpha_1(p_1-2)}\cdot \omega^{\frac12}\bigg(-
\frac{\omega}{2\alpha_1}\bigg)^\frac{4 - p_1}{p_1 - 2} =
\omega^{\frac{6-p_1}{2(p_1-2)}}
\end{split}
    \end{equation}
\end{linenomath}
which converges to zero as \(2<p_1<6\).
\end{proof}
\nocite{GSS87}
\nocite{GSS90}
\bibliographystyle{amsplain}
\bibliography{bibliography}
\end{document}